\documentclass[a4paper, 12pt,reqno]{amsart}
\usepackage[T1]{fontenc}
\usepackage{amsmath,amssymb,amsthm,eucal,manfnt,amsfonts, amsmath,amscd, amssymb, latexsym, mathrsfs, verbatim, wasysym,marvosym}%stmaryrd, ifsym
\usepackage{color, tikz}
\usepackage[all]{xy}
\usepackage{cite}
\usepackage{sidecap}
\usepackage{tikz}
\usepackage{graphicx}
\usepackage{lmodern}
\usepackage[colorlinks]{hyperref}
\usepackage{verbatim}
\usepackage[colorinlistoftodos]{todonotes}
\usepackage{braket,bm,enumitem}
\usepackage{mathtools}
\usepackage{stackrel}
\usepackage{booktabs}
\usepackage{multirow}

\DeclareRobustCommand{\rchi}{{\mathpalette\irchi\relax}}
\newcommand{\irchi}[2]{\raisebox{\depth}{$#1\chi$}} % inner command, used by \rchi

\renewcommand{\Set}[1]{\left\{#1\right\}}

\definecolor{antiquewhite}{rgb}{0.98, 0.92, 0.84}

\title[Real Spectral Triples on Crossed Products]{Real Spectral Triples on Crossed Products}

\author{Alessandro Rubin and Ludwik D\k{a}browski}
\date{}
\email{alrubin@sissa.it, dabrowski@sissa.it }
\address{Scuola Internazionale Superiore di Studi Avanzati (SISSA), via Bonomea 265, I-34136 Trieste}

\keywords{}
\thanks{}
\subjclass[]{}

\advance\topmargin by -\headheight
\advance\topmargin by -\headsep
\numberwithin{equation}{section} %% needs `amsmath' package

\theoremstyle{plain} %% needs `amsmath' package
\newtheorem{thm}{Theorem}[section]
\newtheorem{lemma}[thm]{Lemma}
\newtheorem{prop}[thm]{Proposition}

\newtheorem{ass}{Assumptions}

\theoremstyle{definition} %% needs `amsmath' package
\newtheorem{defn}[thm]{Definition}
\newtheorem{exa}[thm]{Example}

\newtheorem*{conv}{Conventions}
\newtheorem*{ack}{Acknowledgments}

\theoremstyle{remark} %% needs `amsmath' package
\newtheorem{rmk}[thm]{Remark}

\DeclareMathOperator{\Dom}{Dom}   %% domain of an operator
\DeclareMathOperator{\linspan}{span} %% linear span
\newcommand{\C}{\mathbb{C}}   %% complex numbers
\newcommand{\norm}[1]{\left\lVert#1\right\rVert}

\newcommand\xqed[1]{%
	\leavevmode\unskip\penalty9999 \hbox{}\nobreak\hfill
	\quad\hbox{#1}}
\newcommand\demo{\xqed{$\Box$}}

\usepackage{color}
\definecolor{darkgreen}{cmyk}{1,0,1,.2}
\definecolor{m}{rgb}{1,0.1,1}
\definecolor{green}{cmyk}{1,0,1,0}
\definecolor{darkred}{rgb}{0.55, 0.0, 0.0}
\definecolor{test}{rgb}{1,0,0}
\definecolor{cmyk}{cmyk}{0,1,1,0}

\begin{document}

\begin{abstract}
Given a spectral triple on a unital $C^{*}$-algebra $A$ and an equicontinuous action of a discrete group $G$ on $A$, a spectral triple on the reduced crossed product $C^{*}$-algebra $A\rtimes_{r} G$ was constructed by Hawkins, Skalski, White and Zacharias in \cite{hawkins2013spectral},  extending the 
construction by Belissard, Marcolli and Reihani in \cite{bellissard2010dynamical}, by using the Kasparov product to make an ansatz for the Dirac operator. Supposing that the triple on $A$ is equivariant for an action of $G$, we show that the triple on $A\rtimes_rG$ is equivariant for the dual coaction of $G$. If moreover an equivariant real structure $J$ is given for the triple on $A$, we give constructions for two inequivalent real structures on the triple $A\rtimes_rG$. We compute the KO-dimension with respect to each real structure in terms of the KO-dimension of $J$ and show that the first and the second order conditions are preserved. Lastly, we characterise an equivariant orientation cycle on the triple on $A\rtimes_rG$ coming from an equivariant orientation cycle on the triple on $A$. We show, along the paper, that our constructions generalize the respective constructions of the equivariant spectral triple on the noncommutative $2$-torus.

\end{abstract}

\maketitle

\tableofcontents
%
%
%\parindent=0pt
%\parskip=6pt
%\parindent=0pt

%%%%%%%%%%%%%%%%%%%%%%%%%%%%%%%%%%%%%%%%%%%%%%%%%%%%%%%%%%%%%%%%%%%%%%%%%%%%%%%%%
\section{Introduction}
The notion of a ``spectral triple'' (also known as unbounded Fredholm module, or unbounded K-cycle) was introduced by Alain Connes in the course of studying a noncommutative generalization of the Atiyah-Singer index theorem. Its prototype is given by the commutative ${*}$-algebra $C^\infty(M)$ of smooth  
functions on a compact spin manifold $M$ and the Dirac operator on the Hilbert space of square-integrable spinors. Under a few additional assumptions, any commutative spectral triple must be of this form and so it is possible to recover the original manifold from these data \cite{connes1996gravity,connes2013spectral}. 
So far, most research has focused on investigating the properties of particular known spectral triples. However, despite their importance and extensive  study, it is not yet fully understood under what conditions it is possible to define a spectral triple on a given $C^{*}$-algebra; various examples have been constructed only for some specific classes of $C^{*}$-algebras, for instance matrix algebras \cite{paschke1998discrete,krajewski1998classification}, group $C^{*}$-algebras of discrete  hyperbolic  groups, $AF$-algebras, algebras arising as $q$-deformations of the function algebra of Lie groups (see references in \cite{hawkins2013spectral}) and algebras with the ergodic action of a compact Lie group \cite{gabriel2013ergodic}. Spectral triples with nontrivial K-homological content have also been constructed on Cuntz-Krieger algebras \cite{goffeng15Cuntz-Krieger}, boundary crossed products by hyperbolic isometry groups \cite{mesland20Hecke} and groupoid $C^{*}$-algebras \cite{Meslandgroupoid}.

In this respect, spectral triples on crossed products $A\rtimes_{r} G$ proffer new noncommutative  examples (even when $A$ and $G$ are abelian) as they can be regarded as noncommutative quotient spaces 
(see e.g., \cite[Chapter 2]{khalkhali2009basic}).

A spectral triple on the crossed product was first introduced in \cite{bellissard2010dynamical} starting from the equicontinuous action of $G=\mathbb{Z}$ on a spectral triple $(\mathcal{A},H,D)$ on a $C^{*}$-algebra
$A$. Therein 
the question studied was how the Dirac operator $D$ encodes the metric properties of a noncommutative space, in order to define  in particular a noncommutative analogue of a (compact) complete metric space (see e.g., \cite{connes1989compact},\cite{rieffel2004compact}).
This  work was generalised in \cite{hawkins2013spectral}, using as building blocks the spectral triples \cite{connes1989compact} on the reduced group $C^{*}$-algebra $C^{*}_{r}(G)$, where the group $G$ is discrete and endowed with a length-type function, and which is then assumed to act smoothly and equicontinuously on the spectral triple $(\mathcal{A},H,D)$. 
As pointed out in \cite{hawkins2013spectral}, the key idea is to use (a representative for the) external unbounded Kasparov product  to produce a spectral triple on the tensor product $A\otimes C^{*}_{r}(G)$ and then check under which conditions the same formula still defines a triple $(C_{c}(G,\mathcal{A}),\widehat{H},\widehat{D})$ on the crossed product. However, the rationale why this construction should bear any relation to the (external) Kasparov product is so far unclear.  Besides the structure of a quantum metric space, 
	in \cite{hawkins2013spectral} it is proved  that the summability of spectral triples is preserved (under some additional assumptions)  under the passage from 
	$(\mathcal{A},H,D)$ to $(C_{c}(G,\mathcal{A}), \widehat{H}, \widehat{D})$,
	and that the {\em dimension} is additive. It is also shown that the non-degeneracy of the triple is maintained (see Definition \ref{90}). \\
	
 In this paper we consider the construction given in  \cite{hawkins2013spectral} and address a few other so-called {\em axioms} of noncommutative manifolds formulated by A. Connes \cite{connes1996gravity}: our main new contributions concern compatibility with real structures, first and second order conditions and orientation cycles.
They are essential in the reconstruction of the geometric data of commutative spectral triples \cite{connes2013spectral}, and for the theory of gauge transformations and gauge fields in case of almost-commutative spectral triples, 
including, e.g., the spectral version of the Standard Model of fundamental particles and their interactions in physics, cf. \cite{connes2019noncommutative}. It is also worth to mention that a spectral triple equipped with the reality structure defines a real unbounded KO-cycle, which couples to the real KO-theory with a finer periodicity modulo 8, rather than modulo 2 in the usual (complex) case.
Our guide example is the spectral triple on the noncommutative $2$-torus \cite{rieffel1981c,gracia2013elements}, regarded as the crossed product $C(S^{1})\rtimes \mathbb{Z}$. We show that our constructions for the real structures and the orientation cycles generalize the usual ones on the noncommutative torus. 

We work under the assumption that
the starting triple $(\mathcal{A},H,D)$ is equivariant with respect to a unitary representation 
$G\ni g\mapsto u_g\in\mathcal{L}(H)$. In this case, Remark\,29 in \cite{hawkins2013spectral},
to which we add a few details,
brings forth another unitarily equivalent spectral triple 
on $A\rtimes_{\alpha,r}G$, which thus has the same KK-class. 
We also assume that $G$ acts by isometries, i.e., the operators $u_g$ commute with the Dirac operator $D$,
and thus the two spectral triples have the same Dirac operator 
$\widehat{D}$ and only different representations of the algebra $A\rtimes_{r}G$ on the Hilbert space $\widehat{H}$.\\

We start in Section \ref{sec1} with some preliminary material on spectral triples, their products and relation to K-homology, real structures  and provide some examples on discrete groups.

In Section \ref{48} we mostly recall the construction developed in \cite{hawkins2013spectral}, both in the so called equivariant and non-equivariant settings, and relate it to Kasparov's bivariant K-theory \cite{kasparov1980operator, gg1988equivariant}. In particular, we prove that the unsatisfactory relation with the external Kasparov product can be better explained through the internal Kasparov product under suitable isomorphisms. 
At the end, we show that $(C_{c}(G,\mathcal{A}),\widehat{H},\widehat{D})$ 
is equivariant under the dual coaction of $G$ (or, more precisely, the coaction of the group algebra $\C G$ which is in fact a Hopf algebra). 

Sections \ref{J} and \ref{c} contain our main results.
In Section \ref{J} we first construct a real structure $\widehat{J}$ on $(C_{c}(G,\mathcal{A}),\widehat{H}, \widehat{D})$ provided $(\mathcal{A},H,D)$ admits a real structure $J$ such that $Ju_{g}=u_{g}J$ for any $g\in G$. If $G$ is abelian, we construct a second inequivalent real structure $\widetilde{J}$ provided $Ju_{g}=u_{g}^{*}J$ for any $g\in G$. 
We show that in both cases we can consider the relationship between $J$ and $u$ as the equivariance of $J$ with respect to the action of $\mathbb{C}G$ endowed with a suitable $*$-structure, and find that both $\widehat{J}$ and $\widetilde{J}$ are equivariant for the dual coaction of $G$. 
Furthermore, we compute the KO-dimension of $\widehat{J}$ and 
$\widetilde{J}$ in terms of the KO-dimension of $J$, 
and show  that the first and the second order conditions are preserved by the crossed product for both of them,  under suitable conditions. 
Lastly, in Section \ref{c}, using a suitably twisted shuffle product between Hochschild cycles, we induce an equivariant orientation cycle on $(C_{c}(G,\mathcal{A}),\widehat{H}, \widehat{D})$  from an equivariant orientation cycle on $(\mathcal{A},H,D)$.
	
	\begin{conv}
		All Hilbert spaces and $C^{*}$-algebras in this paper are assumed to be separable  and algebras are assumed to be complex and unital. Furthermore, every  group $G$ is assumed to be topological with a locally compact second-countable topology. 
	\end{conv}
	
	\begin{ack} The authors would  like to thank the referees for their  helpful and detailed comments
and suggestions for improvements.
They are also very grateful to P. Antonini for 
several useful comments and to A. Magee for reading the manuscript and helpful remarks.  A.R. thanks A. Magee also for numerous discussions.
	\end{ack}

%%%%%%%%%%%%%%%%%%%%%%%%%%%%%%%%%%%%%%%%%%%%%%%%%%%%%%%%%%%%%%%%%%%%%%%%%%%%%%%%%
\section{Preliminaries: Spectral Triples and Real Structures}\label{sec1}

In this section we recall some basic well known definitions, facts and examples about spectral triples and real structures. For more details we refer to \cite{connes1994noncommutative,gracia2013elements,connes1995noncommutative,connes1996gravity,connes2013spectral,connes2019noncommutative}.

\begin{defn}
An \textit{odd spectral triple} $(\mathcal{A}, H, D)$ on a unital $C^{*}$-algebra $A$ consists of a dense $*$-subalgebra $\mathcal{A}\subseteq A$ represented  by $\pi\colon A\rightarrow \mathcal{L}(H)$ on a Hilbert space $H$ and a self-adjoint operator $D$ (called a \emph{Dirac operator})
densely defined on $\Dom{D}\subset H$  such that 
$(1+D^2)^{-\frac{1}{2}}$ is compact, 
$\pi(a)(\Dom{D})\subseteq \Dom{D}$ and the commutator
$[D,\pi(a)]$ extends to a bounded operator on $H$ for every  
$a\in \mathcal{A}$.
\end{defn}

If the representation $\pi$ is not clear from the context, we will use the notation  $(A, H, D, \pi)$. 
Note that the largest algebra $\mathcal{A}$ with the properties as above is the \emph{Lipschitz algebra} 
$C^{Lip}(A)$, which is a Banach $*$-algebra complete in the norm $\norm{a}_{1} = \norm{a} + \norm{[D,\pi(a)]}$
\cite[Lemma 1]{bellissard2010dynamical}. 
\begin{defn}\label{90}
A spectral triple $(\mathcal{A}, H, D)$ on a unital $C^{*}$-algebra $A$ is called  \emph{non-degenerate} when it is the case that the representation $\pi$ is faithful and $[D, \pi(a)]=0$ for $a\in \mathcal{A}$,
if and only if $a\in \mathbb{C}1_{A}$. 
\end{defn} 

An \emph{even} spectral triple on $A$ is given by the same data with the addition of a $\mathbb{Z}_{2}$-grading, namely a self-adjoint operator $\rchi \colon H\rightarrow H$ called a \emph{grading operator} such that
$\chi^{2}=1_{H}$,
$\pi(a)\rchi = \rchi \pi(a)$ for all $a\in A$, 
$\rchi(\Dom{D})\subseteq \Dom{D}$ and $D\rchi = -\rchi D$.

In the case of an even spectral triple, it is always possible to fix a basis of $H$ in such a way that $H= H_{0}\oplus H_{1}$  and
\begin{displaymath}
\rchi=\left(\begin{matrix}
1 & 0 \\
0 & -1
\end{matrix}\right),\qquad 
\pi =  \left(\begin{matrix}
\pi_{0} & 0 \\
0 & \pi_{1}
\end{matrix}\right), \qquad D= \left(\begin{matrix}
0 & D_{0} \\
D_{0}^{*} & 0
\end{matrix}\right).
\end{displaymath}
Sometimes it is useful to think of odd spectral triples as even triples with the grading $\rchi = \textup{id}_{H}$. In
	this way, it is possible to consider the two situations at the same time.

Generalizing the charge conjugation operator on Dirac  spinors on $\textup{spin}_{c}$ manifolds, A.\,Connes  introduced the following notion:
\begin{defn}[cf.\cite{connes2019noncommutative,connes1996gravity}]\label{67}
	A \emph{real structure} for an (even or odd) spectral triple $(A,H,D, \pi,\rchi)$ is an anti-linear isometry $J\colon H \rightarrow H$ such that 
	\begin{enumerate}
		\item $	\left[\pi(a), J\pi(b)J^{-1}\right] =0$ for all $a,b\in A$ (\emph{zeroth order condition})  
		\item there are signs $\varepsilon, \varepsilon^{\prime} ,\varepsilon^{\prime\prime}=\pm 1$ for which
		\begin{displaymath}
			J^{2} = \varepsilon  \qquad DJ = \varepsilon^{\prime} JD \qquad 		\rchi J = \varepsilon^{\prime\prime} J\rchi .
		\end{displaymath}
	\end{enumerate}
	In this case $(A,H,D,J)$ will be called a \emph{real spectral triple}. 
\end{defn}
There are eight possible triples of signs 
$(\varepsilon, \varepsilon^{\prime}, \varepsilon^{\prime\prime})$ and they determine the so called \emph{KO-dimension} $n\in \mathbb{Z}_{8}$ of the real spectral triple according to the following table\footnote{If $n$ is even,  $J'=J\rchi$ can be used as another real structure with 
the new signs $(\varepsilon\varepsilon^{\prime\prime}, 
-\varepsilon^{\prime}, \varepsilon^{\prime\prime})$
\cite{dkabrowski2011product}.}:
\begin{table}[h]
	\centering
	\begin{tabular}{c|cccccccc}
		\toprule
		$ n $                 &$0$   & $1$   & $2$   & $3$   & $4$ & $5$ &$6$  &$7$    \\ 
		\midrule
		$\varepsilon$                  & $+1$  & $+1$ &  $-1$ & $-1$  &$-1$ &$-1$ &$+1$ &$+1$  \\ 
		$\varepsilon^{\prime}$         &  $+1$ & $-1$  & $+1$  & $+1$  &$+1$ & $-1$& $+1$& $+1$ \\ 
		$\varepsilon^{\prime\prime}$   &$+1$   &     & $-1$  &       & $+1$&      &$-1$ &     \\
		\bottomrule
	\end{tabular}
	%\caption{Dependence of $(\varepsilon, \varepsilon^{\prime}, \varepsilon^{\prime\prime})$ on the KO-dimension.}	
	%\label{145}
\end{table}

Accordingly, even triples have even KO-dimension and odd triples have odd KO-dimension. The  \emph{zeroth order condition} transforms the Hilbert space $H$ into a bimodule over $A$ thanks to the right action 
\begin{displaymath}
	\psi\cdot b = J\pi(b^{*})J^{-1}\psi,  \qquad \psi\in H, b\in A.
\end{displaymath}
Furthermore, the following property is the noncommuative analogue of requiring D to be a first order differential operator.

\begin{defn}[cf.\cite{connes1995noncommutative}]
A real spectral triple $(\mathcal{A},H,D,J)$ on a unital $C^{*}$-algebra $A$ satisfies  the \emph{first order condition} if
	\begin{displaymath}
		\left[\left[D,\pi(a)\right], J\pi(b)J^{-1}\right] = 0.
	\end{displaymath}
for every $a,b\in A$.
	 \end{defn} 
Motivated by the properties of the real structure operator on the spectral triple of the noncommutative Standard Model of particle physics, we consider also the following property.
\begin{defn}[cf.\cite{boyle2014non}]
A real spectral triple $(\mathcal{A},H,D,J)$ on a unital $C^{*}$-algebra $A$ satisfies  the \emph{second order condition} if 
	\begin{displaymath}
	\left[\left[D,\pi(a)\right], J[D,\pi(b)]J^{-1}\right] = 0.
\end{displaymath} 
for every $a,b\in A$.
\end{defn} 

\subsection{Products of Spectral Triples}\indent 

An even or odd spectral triple $(\mathcal{A},H,D)$ over a unital $C^{*}$-algebra $A$ canonically defines a Fredholm module $(A,H,\mathfrak{b}(D))$, where 
\begin{displaymath}
	\mathfrak{b}(D) = D(1+D^{2})^{-1}
\end{displaymath}
is a bounded operator on $H$ called the \emph{bounded transform} of $D$ \cite{connes1994noncommutative}. In particular, the bounded transform induces a map from the space of spectral triples on $A$ to the K-homology group $K^{\bullet}(A)$; this map turns out to be surjective \cite{baaj1983theorie} and so any K-homology class admits an unbounded representative.  The external Kasparov product of two Fredholm modules can be described in terms of their unbounded representatives as follows.

 Given two odd spectral triples $(\mathcal{A}, H_{A}, D_{A})$ and $(\mathcal{B},H_{B}, D_{B})$ on $C^{*}$-algebras $A$ and $B$, let $\mathcal{A}\odot \mathcal{B}$ denote the algebraic tensor product of $\mathcal{A}$ and $\mathcal{B}$. 
Set $H=H_{A}\otimes H_{B}$ and define the operator $D$ on $H\oplus H$ by
\begin{equation}
D=\left(\begin{matrix}
0 & D_{A}\otimes 1 - i\otimes D_{B} \\
 D_{A}\otimes 1 + i\otimes D_{B} & 0
\end{matrix}\right).
\end{equation}
\begin{prop}[cf.\cite{connes1994noncommutative}, Chapter 4.A]\label{prop1}
The triple	$(\mathcal{A}\odot \mathcal{B}, H\oplus H, D)$ is an even spectral triple on $A\otimes_{min} B$ and represents the external Kasparov product of $[(\mathcal{A}, H_{A}, D_{A})]\in KK_{1}(A,\mathbb{C})$ with  $[(\mathcal{B}, H_{B}, D_{B})]\in KK_{1}(B,\mathbb{C})$.  
\end{prop}

Analogously, 
given an even spectral triple $\left(\mathcal{A}, H_{A,0}\oplus H_{A,1}, D_{A}= \left(\begin{matrix}
0 & D_{0} \\
D_{0}^{*} & 0
\end{matrix}\right)\right)$ on $A$ and an odd spectral triple $\left(\mathcal{B},H_{B}, D_{B}\right)$ on $B$, set $H=(H_{A,0}\otimes H_{B})\oplus (H_{A,1}\otimes H_{B})$ and define the operator  $D$ on $H$ by
\begin{equation}
D=\left(\begin{matrix}
1\otimes D_{B} & D_{0}\otimes 1  \\
D_{0}^{*}\otimes 1  & -1\otimes D_{B}
\end{matrix}\right).
\end{equation}
\begin{prop}[cf.\cite{higson2000analytic}, Chapter 9]\label{prop}
The triple	$(\mathcal{A}\odot \mathcal{B}, H, D)$ is an odd spectral triple on $A\otimes_{min} B$ and represents the external Kasparov product of $[(\mathcal{A}, H_{A}, D_{A})]\in KK_{0}(A,\mathbb{C})$ with  $[(\mathcal{B}, H_{B}, D_{B})]\in KK_{1}(B,\mathbb{C})$.   
\end{prop}

\subsection{Spectral Triples over Discrete Groups}\label{Appendix 2}\indent

As we will see in the following, a crucial ingredient in our construction of a triple on a crossed product $A\rtimes_{\alpha,r}G$ is a spectral triple on the reduced group algebra $C^{*}_{r}(G)$. It is known that unbounded Kasparov modules on groups and groupoids exist when they are endowed with a weight type function, see e.g. \cite{connes1989compact, bertozzini2006category,austad21modulation,Meslandgroupoid }. In this subsection we recall some basic facts and definitions specializing the discussion for our purposes. 

\begin{defn}
	A \emph{weight} on a group $G$ is a function $l\colon G\rightarrow \mathbb{R}$. A weight is \emph{proper} if the level sets $\Set{g\in G \mid  -n \leq l(g) \leq n}$ are finite for each $n\in \mathbb{N}$. We say that a weight is \emph{non-degenerate} when $l(g)=0$ if and only if $g=e$. 
\end{defn}

	Proper weights can exist only on countable groups for obvious reasons. Since we assume that all topological groups $G$ are  second countable, every discrete group is in particular countable and may admit proper weights.

\begin{exa}\label{exa}Any group homomorphism  $l\colon G\rightarrow \mathbb{R}$ is clearly an example of a weight on $G$. In this case $l$ must be cyclic in the sense that
	\begin{displaymath}
	l(gh) = l(hg) \qquad \forall g,h\in G
	\end{displaymath}
	 as $\mathbb{R}$ is abelian. Note that if the homomorphism $l$ is non-degenerate as a weight, then $G$ must be abelian. \demo
	\end{exa}

\begin{defn}
	A weight $l\colon G\rightarrow \mathbb{R}$ is said to be a \emph{Dirac weight} if for every $g\in G$ the (left) translation function $l_g\colon G\rightarrow \mathbb{R}$ given by 
	\begin{displaymath}
	l_g(x)=l(x)-l(g^{-1}x) \qquad x\in G
	\end{displaymath} is bounded. A Dirac weight $l$ is said to be of \emph{first-order type} if for every $g\in G$ the (left) translation function $l_g$ is constant.   
\end{defn}
	A special case of a weight on a  group $G$ is given by the notion of a \emph{length function}, namely a non-degenerate weight $l\colon G\rightarrow \mathbb{R}$ such that $l(xy)\leq l(x) + l(y)$ for every $x,y\in G$ and $l(x) = l(x^{-1})$ for every $x\in G$.
	Note that any length function is a positive Dirac weight as 
	\begin{displaymath}
		0 = l(e) = l(xx^{-1})\leq l(x) + l(x^{-1}) = 2l(x) \qquad \forall x\in G
	\end{displaymath}
and 
\begin{displaymath}
		l_{g}(x) = l(x) - l(g^{-1}x)\leq l(g) \qquad \forall x,g\in G.
\end{displaymath}
 The prototypical length function is the \emph{word metric} on a finitely generated group $G$ which associates  to any element $g\in G$ the minimum number of generators needed to write $g$ (for an \emph{a priori} fixed generating set).   

\begin{lemma}[cf.\cite{bertozzini2006category}, Lemma 3.13]\label{68}
	Let $l\colon G\rightarrow \mathbb{R}$ be a weight on a group $G$. The following facts are equivalent:
	\begin{enumerate}
		\item $l$ is a first-order type Dirac weight
		\item $l=\alpha + \varphi$ where $\alpha$ is constant and $\varphi\colon G\rightarrow \mathbb{R}$ is a homomorphism
		\item $l(xzy^{-1}) - l(zy^{-1}) = l(xz) - l(z)$ for every $x,y,z\in G$. 
	\end{enumerate}
\end{lemma}

Till the end of this subsection we shall consider a discrete group $G$ endowed with a proper Dirac weight  $l\colon G\rightarrow \mathbb{R}$. 
Let $M_{l}$ be the multiplication operator by $l$ on the domain of finitely supported elements of $\ell^2 (G)$ and let us denote  also by $M_{l}$ its self adjoint extension to $\ell^{2}(G)$. The group algebra $\mathbb{C}G$ acts on $\ell^{2}(G)$ via the left regular representation $\lambda_{g}\delta_{h}=\delta_{gh}$ for $g,h\in G$, and the data 
\begin{equation}\label{28}
(\mathbb{C}G, \ell^2 (G), M_l,\lambda)
\end{equation} form an odd spectral triple  on the reduced group $C^{*}$-algebra $C^*_r(G)$. Indeed, $M_{l}$ is self-adjoint  as $l$ takes values in $\mathbb{R}$, the properness of $l$ implies that the resolvent of $M_{l}$ is compact and the fact that $l$ is a Dirac weight guarantees that the commutators $[M_{l}, \lambda_{g}]= M_{l_{g}}\lambda_{g}$ are bounded for every $g\in G$. Note further that if the weight $l$ is non-degenerate, then the triple \eqref{28} is also non degenerate. 

\begin{exa}Consider the discrete group $G=\mathbb{Z}$ endowed with the non-degenerate proper Dirac weight $\imath\colon \mathbb{Z}\rightarrow \mathbb{R}$ given by the inclusion. It is well known (see e.g. \cite[pag. 240]{Meslandgroupoid}) that the spectral triple as defined in \eqref{28} agrees with the usual spectral triple on $C(S^{1})$ arising from the Dirac operator $-i\frac{d}{d\theta}$ under the Fourier transform $C^{*}_{r}(\mathbb{Z})\simeq C(S^{1})$. Note in particular that, at the level of K-homology, it is a generator of the cyclic group $K^{1}(C^{*}_{r}(\mathbb{Z}))\simeq \mathbb{Z}$. 
\end{exa}

 \begin{rmk}\label{72bis}We thank the referee for this remark. 	The properness condition for a Dirac weight on a group is \emph{extremely restrictive} as it forces the weight to grow (for example, there is not any constant proper weight  on a group unless the group is finite). Furthermore, the space of proper weights on a group may be a priori empty. To avoid this condition, one is somewhat forced to fall into Kasparov's KK-theory (e.g. \cite[Section 3.7]{Meslandgroupoid}). However, avoiding this condition is not always convenient: for example, any positive weight will yield a spectral triple whose	K-homology class is trivial. 
 \end{rmk}

 We now pass to the study of real structures on the spectral triple \eqref{28}. 
\begin{prop}[cf. \cite{rieffel2002group,bertozzini2006category}]\label{30}
	Let $G$ be a discrete group and $l\colon G\rightarrow \mathbb{R}$ a proper Dirac weight. The anti-unitary involutive map $J_{G}\colon \ell^{2}(G)\rightarrow \ell^{2}(G)$
	given by the anti-linear extension of
	\begin{equation}\label{29}
	J_{G}\delta_{g} = \delta_{g^{-1}} \qquad \forall g\in G
	\end{equation} 
	is a real structure on the odd spectral triple \eqref{28} if and only if for every $g\in G$, $l(g^{-1}) = \varepsilon^{\prime}l(g)$, where $\varepsilon^{\prime}=\pm 1$. 	In this case the KO-dimension of the real structure is given by the pair $(+1,\varepsilon^{\prime})$ and can be either $1$ or $7$. 
 
\end{prop}

\begin{proof}
Clearly $J_{G}^{2}=1$. Then just note that the equation $M_{l}J_{G} = \varepsilon^{\prime}J_{G}M_{l}$ is fulfilled if and only if $l(g^{-1}) = \varepsilon^{\prime}l(g)$ for every $g\in G$. 
\end{proof}

\begin{prop}[cf.\cite{rieffel2002group}]\label{72}
	Let $G$ be a discrete group and $l\colon G\rightarrow \mathbb{R}$ a proper Dirac weight. Suppose that the map $J_{G}$ given in \eqref{29} is a real structure for the spectral triple \eqref{28}. Then $(\mathbb{C}G,\ell^{2}(G),M_{l},J_{G})$ satisfies the first order condition if and only if $l$ is either a constant or a homomorphism. 
\end{prop}

\begin{proof}An easy computation shows that the first order condition holds true if and only if we have that
	\begin{displaymath}
	l(xzy^{-1}) - l(zy^{-1}) = l(xz) - l(z)
	\end{displaymath}
for every $x,y,z\in G$. By Lemma \ref{68}, this means that $l$ must be the sum of a constant and a homomorphism. But from Proposition \ref{30} we know that $l(g^{-1}) = \varepsilon^{\prime}l(g)$ for every $g\in G$ and it is easy to see that if $\varepsilon^{\prime}=1$ then $l$ must be constant and that if $\varepsilon^{\prime}=-1$ then $l$ must be a homomorphism. 
\end{proof}

\color{black}
\begin{prop}
	Let $G$ be a discrete group and $l\colon G\rightarrow \mathbb{R}$ a proper Dirac weight. Suppose that the map $J_{G}$ given in \eqref{29} is a real structure for the spectral triple \eqref{28}. If $(\mathbb{C}G,\ell^{2}(G),M_{l},J_{G})$ satisfies the first order condition, then it satisfies the second order condition.  
 
\end{prop}
	\begin{proof}
		An easy computation shows that for any $g,h,x\in G$ we have
		\begin{displaymath}
		\begin{split}
		\left[[M_{l}, \lambda_{h}], J_{G}[M_{l}, \lambda_{g}]J_{G}^{-1}\right] \delta_{x}&= \left(l(gx^{-1}) - l(x^{-1})\right)\left(l(hxg^{-1}) - l(xg^{-1})\right)\delta_{hxg^{-1}} \\
		& \qquad \qquad  - \left(l(hx) - l(x)\right)\left(l(gx^{-1}) - l(x^{-1})\right)\delta_{hxg^{-1}}.
		\end{split}
		\end{displaymath}
		The thesis comes from Proposition \ref{72}. 
	\end{proof}

%%%%%%%%%%%%%%%%%%%%%%%%%%%%%%%%%%%%%%%%%%%%%%%%%%%%%%%%%%%%%%%%%%%%%%%%%%%%%%%%%
\section{Spectral Triples on Crossed Products}\label{48}
 
In this section we  recall (with some additions) the relevant results from \cite{hawkins2013spectral}.

\subsection{Spectral Triples Arising From Equicontinuous Actions}\label{47}	\indent

\begin{ass}\label{assumption1}
	Let $(\mathcal{A}, H, D,\pi)$ be an odd spectral triple  on a unital $C^*$-algebra $A$ with $\pi$ faithful and $(A,G,\alpha)$ a $C^{*}$-dynamical system, namely a discrete group $G$ acting on $A$ by automorphisms such that the map $x\mapsto \alpha_{x}(a)$ from $G$ to $A$ is continuous  for any $a\in A$. Let us suppose further that $G$ is equipped with a proper Dirac weight $l\colon G\rightarrow \mathbb{R}$ so that this determines a spectral triple $(\mathbb{C}G, \ell^{2}(G), M_{l}, \lambda)$ over $C^{*}_{r}(G)$ (see Subsection \ref{Appendix 2}).
\end{ass}

 We regard an element $f\in C_{c}(G,\mathcal{A})$ as a sum $f=\sum_{g\in G} a_{g}\delta_{g}$ where only finitely many elements $a_{g}$ are non-zero and $\delta_{g}$ is the function which is $1$ on $g\in G$ and zero otherwise. We define a product and  $*$-operation on $C_{c}(G,\mathcal{A})$ by
	\begin{displaymath}
		\delta_{h}\delta_{g} = \delta_{hg} \quad \delta_{g}^{*} = \delta_{g^{-1}} \quad \delta_{g}a\delta_{g}^{*} = \alpha_{g}(a),
	\end{displaymath}
	such that $C_{c}(G,\mathcal{A})$ becomes a $*$-algebra. The pair of maps 
	\begin{equation}\label{73}
		\begin{dcases}
			\hat{\pi}_{1}(a)(\xi\otimes \delta_{g}) = \pi\left(\alpha_{g^{-1}}(a)\right)\xi\otimes \delta_{g} \\
			\hat{\lambda}_{h}(\xi\otimes \delta_{g}) = \xi\otimes \delta_{hg},
		\end{dcases}
	\end{equation}
	where $a\in A$, $\xi\in H$ and $g,h\in G$, provide representations of $A$ and $G$ on $H\otimes \ell^{2}(G)$ which are \emph{covariant}, namely such that
	\begin{equation}\label{60}
		\hat{\lambda}_{h}\hat{\pi}_{1}(a)\hat{\lambda}_{h}^{*} = 	\hat{\pi}_{1}(\alpha_{g}(a)).
	\end{equation}
	Thanks to \eqref{60}, the \emph{integrated form} $\hat{\pi}_{1}\rtimes \hat{\lambda}$ of the pair $(\hat{\pi}_{1},\hat{\lambda})$ given by
	\begin{equation}\label{61}
		\hat{\pi}_{1}\rtimes \hat{\lambda}(a\delta_{g}) = \hat{\pi}_{1}(a) \hat{\lambda}_{g}
	\end{equation}
	is a $*$-representation of $C_{c}(G,\mathcal{A})$ on $H\otimes \ell^{2}(G)$. We define the reduced crossed product $A\rtimes_{\alpha,r}G$ as the completion of $C_{c}(G,\mathcal{A})$ under the norm given by regarding its elements as bounded operators on that Hilbert space. \\

We now return to the discussion of the spectral triple on $A\rtimes_{r,\alpha}G$. 
Following Proposition \ref{prop1}, we define a Dirac operator $\widehat{D}$ on $\widehat{H} = H\otimes \ell^{2}(G)\otimes \mathbb{C}^{2}$ by 
\begin{equation}\label{8}
\begin{split}
\widehat{D} &= D\otimes 1\otimes \sigma_1 + 1\otimes M_{l}\otimes \sigma_{2}\\
&= \left( \begin{matrix}
0 & D\otimes 1 - i \otimes M_{l} \\
D\otimes 1 + i \otimes M_{l} & 0 \\
\end{matrix}\right),
\end{split}
\end{equation}
where $\sigma_{1}$ and $\sigma_{2}$ are Pauli matrices and we consider $A\rtimes_{r,\alpha}G$ acting diagonally on $\widehat{H}$ by two copies of the integrated form \eqref{61}. The operator \eqref{8} is clearly densely defined and self-adjoint with compact resolvent. The only non-trivial fact is to check that the commutator of $\widehat{D}$ with the representation of $A\rtimes_{\alpha, r} G$ is bounded: on the one hand, we have that
\begin{equation}\label{41}
	\left[1\otimes M_{l}, \hat{\pi}_{1}(a)\hat{\lambda}_{h}\right] = (1\otimes M_{l_{h}})\hat{\pi}_{1}(a)\hat{\lambda}_{h}
\end{equation}
for any $h\in G$. As $l$ is a Dirac weight, the operator  $M_{l_{h}}$ of multiplication by the function $l_{h}(x) = l(x)  - l(h^{-1}x)$ on $\ell^{2}(G)$ is bounded. On the other hand, for $a\in \mathcal{A}$ and $g,h\in G$, we have
\begin{equation}\label{40}
	\left[D\otimes 1, \hat{\pi}_{1}(a)\hat{\lambda}_{h}\right](\xi\otimes \delta_{g}) = \left[D, \pi(\alpha_{g^{-1}h^{-1}}(a))\right]\xi\otimes \delta_{hg}.
\end{equation}
A priori this quantity is not well defined as $\pi(\alpha_{g}(a))$ might not belong to the domain of $D$; and even in that case the right hand side is not necessarily bounded. We overcome these difficulties in the following way.

\begin{defn}[cf.\cite{hawkins2013spectral}]
	We say that the (continuous) action $\alpha\colon G\rightarrow \textup{Aut}(A)$ is:
	\begin{itemize}
		\item \emph{smooth} if $\alpha_g(\mathcal{A})\subseteq \mathcal{A}$ for all $g\in G$
		\item \emph{equicontinuous} if for all $a\in \mathcal{A}$
		\begin{equation}\label{81}
			\sup_{g\in G}\norm{[D,\pi(\alpha_g(a))]}< \infty.	
		\end{equation}  
	\end{itemize}
\end{defn}

\begin{thm}[cf.\cite{hawkins2013spectral}]\label{thm}
		 Let $(\mathcal{A}, H, D)$ and $G$ as in Assumptions \ref{assumption1}. Suppose further that $G$ acts smoothly and equicontinuously on $A$. The triple $(C_{c}(G,\mathcal{A}), \widehat{H}, \widehat{D}, \hat{\pi}_{1}\rtimes \hat{\lambda})$ is an even spectral triple on $A\rtimes_{\alpha,r} G$. Furthermore, if the weight $l$ is non-degenerate and the triple  $(\mathcal{A}, H, D)$ is non-degenerate, then the triple 
$(C_{c}(G,\mathcal{A}), \widehat{H},\widehat{D},\hat{\pi}_{1}\rtimes \hat{\lambda})$ is also non-degenerate.
\end{thm}

\begin{rmk}
	Even though the operator $\widehat{D}$ resembles (a representative for) the external Kasparov product of $D$ and $M_{l}$, the triple $(C_{c}(G,\mathcal{A}), \widehat{H}, \widehat{D})$ is \emph{not} the external Kasparov product of the triple $(A, H, D)$ with the triple $(C^{*}_{r}(G), \ell^{2}(G), M_{l})$ as the isomorphic vector spaces $A\rtimes_{\alpha,r}G$ and $A\otimes C^{*}_{r}(G)$ are in general not isomorphic as algebras.
\end{rmk}

It is well known that the construction in Theorem \ref{thm} can be thought as a generalization of the boundary map in the K-homology Pimsner-Voiculescu sequence (cf. \cite{pimsner1980exact}) in the following sense. Applying the six-term exact sequence in K-homology to the generalised Toeplitz extension 
\begin{equation}\label{23}
0\longrightarrow A\otimes \mathcal{K}(\ell^{2}(\mathbb{Z})) \longrightarrow T_{\alpha}\longrightarrow A\rtimes_{\alpha}\mathbb{Z}\longrightarrow 0
\end{equation}
and using $K^{i}(T_{\alpha})\simeq K^{i}(A)$, one obtains the Pimsner-Voiculescu exact sequence 
\begin{center}
	\begin{tikzpicture}[scale=2.5]
	\node (C) at (0,0) {$K^{0}(A)$};
	\node (B) at (1.5,0) {$K^{0}(A)$};
	\node (A) at (3,0) {$K^{0}(A\rtimes_{\alpha}\mathbb{Z})$};
	\node (F) at (3,-1) {$K^{1}(A)$.};
	\node (E) at (1.5,-1) {$K^{1}(A)$};
	\node (D) at (0,-1) {$K^{1}(A\rtimes_{\alpha}\mathbb{Z})$};
	\path[-latex,font=\scriptsize]
	(A) edge node[above]{$\varepsilon$} (B)
	(B) edge node[above]{$1-\alpha^{*}$} (C)
	(C) edge node[left]{$\partial^{0}$} (D)
	(D) edge node[above]{$\varepsilon $} (E)
	(E) edge node[above]{$1-\alpha^{*}$} (F)
	(F) edge node[right]{$\partial^{1}$} (A);
	\end{tikzpicture}
\end{center}
Here $\varepsilon$ is the pull-back map. 
Note that the boundary maps $\partial^{0}$ and $\partial^{1}$ are just the left Kasparov multiplication by the class $[\tau]\in KK_{1}(A\rtimes_{\alpha}\mathbb{Z}, A\otimes \mathcal{K})\simeq KK_{1}(A\rtimes_{\alpha}\mathbb{Z}, A)$ defined by the extension \eqref{23} (see \cite[$\S 19.5-6$]{blackadar1998k}). 

Consider now the discrete group $G=\mathbb{Z}$ and the Dirac weight $\imath\colon \mathbb{Z}\rightarrow \mathbb{R}$ given by the inclusion. It is clear that the triple $(A\rtimes_{\alpha}\mathbb{Z}, \widehat{H}, \widehat{D})$  is the Kasparov product of $(A, H,D)$ with the class of the generalised Toeplitz extension \eqref{23}. This proves that:
	\begin{equation}
	\partial^{1} \left[(A, H,D)\right] = \left[(A\rtimes_{\alpha}\mathbb{Z}, \widehat{H}, \widehat{D})\right].
\end{equation}
More refined methods to see this fact can be found in many places, for instance \cite{rennie17KMSstates}.

\begin{rmk}We thank the	 referee for this remark. By passing to higher order spectral triples and applying a logarithmic dampening of the Dirac operator, it is possible to handle also non-equicontinuous actions and get similar results. As an example of this strategy, consider a non-isometric diffeomorphism on the circle which give rise to a twisted spectral triple on $S^{1}$ \cite[Example 1.6]{mesland19untwisting} of which the logarithmic dampening is an ordinary spectral triple \cite[Example 1.9]{mesland19untwisting}. The pullback action of the diffeomorphism generates an action of $\mathbb{Z}$ on the circle which is not equicontinuous. Deforming the standard Dirac operator on the circle and following a prescription similar to that of Theorem \eqref{thm}, one can define a higher order spectral triple on $C(S^{1})\rtimes \mathbb{Z}$ whose K-homology class coincides with the class of its dampening \cite[Corollary 1.41]{mesland19untwisting} and represents the Pimsner-Voiculescu boundary map image of the aformentioned dampened class on $C(S^{1})$ \cite[Prop. 1.30 and the preceeding discussion]{mesland19untwisting}.

\end{rmk}

Let us now discuss the properties of the triple of Theorem \ref{thm}. 
Concerning the dimension axiom \cite{connes1995noncommutative}, recall that a spectral triple $(\mathcal{A}, H,D)$ over the unital $C^{*}$-algebra $A$ is \emph{$p$-summable} for $p>0$ if the operator $(1+D^{2})^{-\frac{p}{2}}$ is trace-class. 
	In \cite{hawkins2013spectral} the additivity of dimension with respect to the crossed product was demonstrated: 

\begin{prop} Let $(\mathcal{A}, H, D)$ and $G$ as in Assumptions \ref{assumption1}. Suppose further that $G$ acts smoothly and equicontinuously on $A$. If the triple $(\mathcal{A}, H, D)$ is $p$-summable and the triple $(\mathbb{C}G, \ell^{2}(G), M_{l})$ is $q$-summable, then
	the triple $(A\rtimes_{\alpha, r}G, \widehat{H}, \widehat{D}, \hat{\pi}_{1}\rtimes \hat{\lambda})$ is $(p+q)$-summable.  
\end{prop}

This follows as in the case of the external product of two spectral triples: 
if $\lambda_{n}$ and $\mu_{m}$ are respectively the sequences of the eigenvalues of $D$ and $M_{l}$, by assumption the sequences $((1+\lambda_{n}^{2})^{-p/2})$ and $((1+\mu_{m}^{2})^{-q/2})$ are convergent. Then, using the inequality
	\begin{displaymath}
(x+y-1)^{\alpha + \beta} \geq x^{\alpha}y^{\beta}, \qquad x,y>1, \, \alpha,\beta>0,
	\end{displaymath}
the double sequence $(1+\lambda_{n}^{2} + \mu_{m}^{2})^{-(p+q)/2}$  also proves convergent. 
	
We pass now to another property.

\begin{defn}
	A spectral triple $(\mathcal{A},H,D)$ is \emph{irreducible} if there is no closed subspace of $H$ invariant under the action of $A$ and $D$.  	
\end{defn}

\begin{prop}Let $(\mathcal{A}, H, D)$ and $G$ as in Assumptions \ref{assumption1}. Suppose further that $G$ acts smoothly and equicontinuously on $A$. If $(A, H, D)$ is irreducible then $(A\rtimes_{\alpha, r}G, \widehat{H}, \widehat{D}, \hat{\pi}_{1}\rtimes \hat{\lambda})$ is also irreducible. 	
\end{prop}

\begin{rmk}\label{34}
	Similarly to what described so far, we can construct odd spectral triples on crossed products starting from even spectral triples and an equicontinuous action in a suitable sense. Let
	\begin{displaymath}
		\left(\mathcal{A},H= H_{0}\oplus H_{1}, D=\left(\begin{matrix}
			0 & D_{0} \\
			D_{0}^{*} & 0
		\end{matrix}\right)\right)
	\end{displaymath} be an even spectral triple on a unital $C^{*}$-algebra $A$ with the $\mathbb{Z}^{2}$-grading  $H_{0}\oplus H_{1}$ and $\pi=\pi_{0}\oplus \pi_{1}$. Let $\alpha$ be an action of a discrete group $G$ on $A$ and let $l\colon G\rightarrow \mathbb{R}$ be a proper Dirac weight. We have a diagonal representation of the reduced crossed product $A\rtimes_{\alpha,r}G$ on $(H_{0}\otimes \ell^{2} (G))\oplus (H_{1}\otimes \ell^{2}(G))$. Provided $\alpha_{g}(\mathcal{A})\subseteq \mathcal{A}$ for all $g\in G$ and the equicontinuity condition
	\begin{displaymath}
		\sup_{g\in G}\norm{\pi_{0}(\alpha_{g}(x))D_{0} - D_{0}\pi_{1}(\alpha_{g}(x))}<+\infty, \qquad \forall x\in \mathcal{A}
	\end{displaymath}
	then the Dirac operator
	\begin{equation}\label{dual1}
		\check{D}=\left(\begin{matrix}
			1\otimes M_{l} & D_{0}\otimes 1 \\
			D_{0}^{*}\otimes 1 & -1\otimes M_{l}
		\end{matrix}\right)
	\end{equation}
	can be used to define an odd spectral triple $(C_{c}(G,\mathcal{A}), \check{H}, \check{D})$ on $A\rtimes_{\alpha,r} G$ for $\check{H} = H\otimes \ell^{2}(G)$. In the same way as for odd spectral triples, we have that this triple:
	\begin{enumerate}
		\item  is non-degenerate whenever the Dirac weight $l$ and the triple $(A,H,D)$ are non-degenerate 
		\item represents the image of the even spectral triple $(A,H,D,\rchi)$ under the Pimsner-Voiculescu boundary map when one considers the group $G=\mathbb{Z}$ and the Dirac weight  given by the inclusion $\imath \colon \mathbb{Z}\rightarrow \mathbb{R}$. \demo
	\end{enumerate} 
	\end{rmk}

\subsection{The Equivariant Construction}\label{sec3.2}\indent

As noticed in \cite[Remark 2.9]{hawkins2013spectral}, there is an alternative though unitarily equivalent (and thus K-homologically equivalent) description of the construction in Theorem \ref{thm} starting from a spectral triple which is equivariant in the following sense. 

\begin{defn}
	Let $(A,G,\alpha)$ be a dynamical system. A spectral triple $(\mathcal{A},H,D)$ on a unital $C^{*}$-algebra $A$ is \emph{equivariant} with respect to the action of $G$, or simply $G$-equivariant, if there exists  a unitary representation $u\colon G\rightarrow \mathcal{L}(H)$ such that:
	\begin{enumerate}
		\item $(\pi, u)$ is a covariant representation of $(A,G,\alpha)$ on $H$.
		\item The operators $u_{g}:=u(g)$  leave the domain of $D$ invariant for all $g\in G$.
		\item The commutator $[u_{g}, D]$  extends to a bounded operator on $H$ for every $g\in G$. This means that the difference 
		\begin{equation}\label{43}
			u_{g}Du_{g}^{*} - D = [u_{g}, D]u_{g}^{*}
		\end{equation}
		is bounded, making the triple an equivariant (unbounded) Kasparov module. 
	\end{enumerate}
	When $[D,u_{g}]=0$ for every $g\in G$  
	we just say that the triple $(\mathcal{A},H,D)$ is $G$-invariant.
\end{defn}

Note that the bounded transform of an equivariant spectral triple defines an equivariant Kasparov module as defined in Appendix \ref{appA} (see \cite{pierrot2006bimodules}). Furthermore, any $G$-equivariant spectral triple $(\mathcal{A},H,D, \pi, u)$ defines a spectral triple $(C_{c}(G,\mathcal{A}), H, D,\pi\rtimes u )$ on the \emph{maximal}  crossed product $A\rtimes_{\alpha}G$ as the commutators
\begin{displaymath}
	[D,\pi(a)u_{g}] = 	[D,\pi(a)]u_{g} + 	\pi(a)[D,u_{g}]
\end{displaymath}
are bounded by hypothesis. Viceversa, using the universal properties of the maximal crossed product, any spectral triple on $A\rtimes_{\alpha} G$ comes from a $G$-equivariant spectral triple. 
We can think of this association as the unbounded version of the well known Green-Julg isomorphism  
\begin{displaymath}
	KK^{G}(A,\mathbb{C})\simeq KK(A\rtimes_{\alpha} G,\mathbb{C})
\end{displaymath} for discrete groups (see \cite{julg1981k} and \cite[Example 4.9]{echterhoff2017bivariant}). 

\begin{exa}[cf. \cite{rieffel1981c,gracia2013elements}]\label{nctorus}
Consider the unit circle $S^{1}$ as the additive group $\mathbb{R}/\mathbb{Z}$ endowed with the quotient structure. We identify functions $f\in C(S^{1})$ with continuous periodic functions  on $\mathbb{R}$ of period $1$. Fix now $\theta\in \mathbb{R}$ and consider the  action of $\mathbb{Z}$ on $C(S^{1})$ given by 
\begin{displaymath}
	\alpha_{n}(f)(t) = f(t+n\theta)
\end{displaymath}	
for $n\in \mathbb{Z}$ and $t\in \mathbb{R}$; it is known  that $(C(S^{1}),\mathbb{Z}, \alpha)$ is a $C^{*}$-dynamical system. The spectral triple $(C(S^{1}), L^{2}(S^{1}), D )$, where $D$ is the self-adjoint extension of the operator $-i\frac{\partial}{\partial x}$ on $L^{2}(S^{1})$, is $\mathbb{Z}$-equivariant with respect to the unitary representation $u$ of $\mathbb{Z}$ on $L^{2}(S^{1})$ given by $u_{n}(f)(t) = f(t+n\theta).$ Note further that $[D,u] =0$ and so the triple is actually $\mathbb{Z}$-invariant. \demo
\end{exa}

Following \cite[Remark 2.9]{hawkins2013spectral}, let us now show another strategy to construct spectral triples on crossed products starting from equivariant spectral triples. 

\begin{ass}\label{assumptions2}
	Let  $(\mathcal{A}, H,D,u)$ be a $G$-equivariant spectral triple  on a unital $C^{*}$-algebra $A$ with $\pi$ faithful and $G$ a discrete group which acts continuously on $A$ by $\alpha$ and is endowed with a proper Dirac weight $l\colon G\rightarrow \mathbb{R}$.
\end{ass}

 The weight $l$  defines an equivariant odd spectral triple $(\mathbb{C}, \ell^{2}(G), M_{l})$ where the group action on $\ell^{2}(G)$ is given by the left regular representation. Note that the commutators $[M_{l}, z]$ are vanishing for all $z\in \mathbb{C}$. The  Kasparov product of $[D]\in KK_{1}^{G}(A,\mathbb{C})$ and $[M_{l}]\in KK_{1}^{G}(\mathbb{C}, \mathbb{C})$ is represented by the even $G$-equivariant  triple 
\begin{equation}\label{25}
(\mathcal{A}, \widehat{H}, \widehat{D}),
\end{equation}
on $A$ where $\widehat{H}=H\otimes\ell^{2}(G)\otimes \mathbb{C}^{2}$, $\widehat{D}$ is defined as in \eqref{8}, the representation of $A$ on $\widehat{H}$ is given by (two copies of) 
\begin{displaymath}
\hat{\pi}_{2}(a) (\xi\otimes \delta_g) = \pi(a)\xi\otimes \delta_g, \qquad a\in A,\xi\in H, g\in G
\end{displaymath} 
and the equivariance is implemented by (two copies of) the representation 
$\hat{\Gamma}\colon G\rightarrow \mathcal{L}(\widehat{H})$ given by 
\begin{equation}\label{gamma}
\hat{\Gamma}_{h}(\xi\otimes \delta_g) = u_{h}\xi\otimes \delta_{hg}
\end{equation} 
for $g,h\in G$. Under the Green-Julg isomorphism, the class of the triple \eqref{25} is represented in $KK(A\rtimes_{\alpha}G, \mathbb{C})$ by the triple 
\begin{equation}\label{25bis}
	(C_{c}(G,\mathcal{A}), \widehat{H}, \widehat{D})
\end{equation} 
where the action of the algebra is now given by (two copies of) the integrated form of the covariant representation $(\hat{\pi}_{2}, \hat{\Gamma})$. 

\begin{exa}
	If we apply this procedure to the equivariant spectral triple of Example \ref{nctorus} where $\mathbb{Z}$ is endowed with the proper weight $i\colon \mathbb{Z}\rightarrow \mathbb{R}$ given by the inclusion, we just get the canonical spectral triple on the noncommutative $2$-torus $C(S^{1})\rtimes_{\alpha}\mathbb{Z}$. Indeed, it is known that the GNS Hilbert space representation $H_{\tau}$ is isomorphic to $L^{2}(S^{1}\times S^{1})$ by mapping the generators of the NC torus $U,V$ to the functions $e^{2\pi i\varphi_{1}}$ and $e^{2\pi i\varphi_{2}}$. Under Fourier transform on the second entry, this isomorphism is then just the map which takes an element $f\otimes \delta_{m}\in C(S^{1})\rtimes_{\alpha}\mathbb{Z}\subseteq H_{\tau}$ and maps it to $f\otimes \delta_{m}$ in $L^{2}(S^{1})\otimes \ell^{2}(\mathbb{Z})$. With this identification, the GNS representation of $C(S^{1})\rtimes \mathbb{Z}$  is just the integrated form of the covariant couple $(\hat{\pi}_{2}, \hat{\Gamma})$, where $\hat{\pi}_{2}$ is the multiplication operator of $C(S^{1})$ and $\hat{\Gamma}$ is defined as in \eqref{gamma} for $u$ as in Example \ref{nctorus}. \demo
	
\end{exa}

Note that, differently from what was discussed in Subsection \ref{47}, the action $\alpha$ in this case need neither be smooth nor equicontinuous in order to ensure a well defined and bounded commutation relation between $\widehat{D}$ and $A\rtimes G$. We have the following result.

\begin{lemma}\label{44}Let $(\mathcal{A}, H,D,u)$ be a $G$-equivariant spectral triple on a unital $C^{*}$-algebra $A$. We have that:
	\begin{enumerate}
		\item  $\alpha_{g}(C^{Lip}(A))\subseteq C^{Lip}(A)$ for any $g\in G$.
		\item If the commutator $[u_{g}, D]$ is uniformly bounded in norm for all $g\in G$, then  the action $\alpha$ of $G$ on $A$ is equicontinuous. 
		\item If  $(\mathcal{A},H,D,u)$ is $G$-invariant, then  the action is Lip-isometric in the sense that 
		\begin{displaymath}
			\norm{[D,\pi(\alpha_{g}(a))]} = 	\norm{[D,\pi(a)]}.	
		\end{displaymath} 
	\end{enumerate}
	
\end{lemma}
\begin{proof}As $\pi(\alpha_{g}(a))=u_{g}\pi(a)u_{g}^{*}$ we have that any element $\alpha_{g}(a)$ for $a\in C^{Lip}(A)$ preserves the domain of $D$.
	Using the covariance it is further easy to see that
	\begin{equation}\label{36}
		[D,\pi(\alpha_{g}(a))] = -\left[ u_{g}[D,u_{g}^{*}], \pi(\alpha_{g}(a))\right] + u_{g}[D,\pi(a)]u_{g}^{*}.
	\end{equation} and so the commutator $[D,\pi(\alpha_{g}(a))]$ is bounded. To prove the second point we note that the action is isometric as $G$ acts by automorphisms. Then from equation \eqref{36} we have
	\begin{displaymath}
		\norm{[D,\pi(\alpha_{g}(a))]}\leq 2\norm{[D,u_{g}]}\norm{\pi(a)} + \norm{[D,\pi(a)]}
	\end{displaymath}
	for any $a\in A$ and $g\in G$. By assumption, the right hand side is uniformly bounded in $g\in G$ and so we have equicontinuity.  Point $(3)$ is an easy consequence of equation \eqref{36}. 
\end{proof}

We now want to show that the spectral triple \eqref{25bis} descends to a triple on the \emph{reduced} crossed product and that, under some extra assumptions, this triple is K-homologically equivalent to the one of Theorem \ref{thm}. Consider the  unitary map $U$ defined on $H\otimes \ell^{2}(G)$ by
\begin{equation}
	U(\xi\otimes \delta_{g}) = u_{g}\xi\otimes \delta_{g}.
\end{equation}
This conjugates the action $\hat{\pi}_{1}\rtimes \hat{\lambda}$  to the action $\hat{\pi}_{2}\rtimes \hat{\Gamma}$ as
\begin{displaymath}
	U\hat{\pi}_{1}(a)U^{*}(\xi\otimes \delta_{g}) = u_{g}\pi(\alpha_{g}^{-1}(a))u_{g}^{*}\xi\otimes \delta_{g} = \pi(a)\xi\otimes \delta_{g} = \hat{\pi}_{2}(a)(\xi\otimes \delta_{g})
\end{displaymath}
and
\begin{displaymath}
	U\hat{\lambda}_{h}U^{*}(\xi\otimes \delta_{g}) = u_{hg}u_{g}^{*}\xi\otimes \delta_{hg} =  \hat{\Gamma}_{h}(\xi\otimes \delta_{g}).
\end{displaymath}
In particular, the representation of the maximal crossed product of the triple \eqref{25bis} is unitarily equivalent to the representation defining the reduced crossed product and so it descends to a representation of $A\rtimes_{\alpha,r}G$.

 The relation between $(C_{c}(G,\mathcal{A}), \widehat{H}, \widehat{D},\hat{\pi}_{2}\rtimes \hat{\Gamma})$ and $(C_{c}(G,\mathcal{A}), \widehat{H}, \widehat{D},\hat{\pi}_{1}\rtimes \hat{\lambda})$ as triples on $A\rtimes_{\alpha,r}G$ is then easy to understand: \emph{if the commutator $[u_{g}, D]$ is uniformly bounded in norm for all $g\in G$, then the triple $(C_{c}(G,\mathcal{A}), \widehat{H}, \widehat{D},\hat{\pi}_{1}\rtimes \hat{\lambda})$  is well defined (by point (2) in Lemma \ref{44}) and defines the same K-homology class of the triple $(C_{c}(G,\mathcal{A}), \widehat{H}, \widehat{D},\hat{\pi}_{2}\rtimes \hat{\Gamma})$ as the unitary $U$ conjugates the representations and $\widehat{D}$ to $(U\oplus U)\widehat{D}(U\oplus U)^{*}$, which is a bounded perturbation of $\widehat{D}$ by hypothesis.}

 As a direct consequence of this discussion we have that, using the Green-Julg isomorphism, if the triple $(\mathcal{A}, H,D)$ is equivariant with $[D,u_{g}]$ uniformly bounded for any $g\in G$ then the triple $(C_{c}(G,\mathcal{A}), \widehat{H}, \widehat{D},\hat{\pi}_{1}\rtimes \hat{\lambda})$ defined in Theorem \ref{thm} represents the \emph{exterior} equivariant Kasparov product of $[D]\in KK^{G}_{1}(A, \mathbb{C})$ with $ [M_{l}]\in KK^{G}_{1}(\mathbb{C}, \mathbb{C})$.

We thank the anonymous referee for pointing out that the spectral triple \eqref{25bis} can be  seen also as a representative of an \emph{interior} Kasparov product as follows. It is known that in bounded KK-theory there is a commutative diagram 
	\begin{figure}[h]
		\centering
		\begin{tikzpicture}{scale =1.4}
			
			\node (A) at (0,0) {$KK^{G}_{1}(A,\mathbb{C})\times KK^{G}_{1}(\mathbb{C},\mathbb{C})$};
			\node (B) at (7,0) {$KK^{G}_{0}(A,\mathbb{C})$};
			\node (C) at (7,-2) {$KK_{0}(A\rtimes_{\alpha} G, \mathbb{C}).$};
			\node (D) at (0,-2) {$KK_{1}(A\rtimes_{\alpha} G,C^{*}(G))\times KK_{1}(C^{*}(G), \mathbb{C})$};

			\path[->,font=\scriptsize,=angle 90]
			(A) edge  node[above]{$\otimes_{\textup{ext}}$} (B)
			(B) edge node[right]{$I_{G}$} (C)
			(A) edge node[right]{$J_{G}\times I_{G}$}(D)
			(D) edge node[above]{$\otimes_{C^{*}(G)}$} (C);
			
		\end{tikzpicture}
	\end{figure}

 Here $J_{G}$ is the Kasparov descent map (see Appendix \ref{appA}), $I_{G}$ is the Green-Julg isomorphism,  $\otimes_{C^{*}(G)}$ represents the interior Kasparov product and $\otimes_{\textup{ext}}$ is the exterior Kasparov product. This is a consequence of the fact that the Kasparov descent is compatible with the product and that it factorizes the Green-Julg isomorphism together with the map
 \begin{displaymath}
 	KK(A\rtimes_{\alpha}G, C^{*}(G))\longrightarrow KK(A\rtimes_{\alpha}G, \mathbb{C})
 \end{displaymath}
 induced by the trivial representation of $G$ (see e.g. \cite[Remark 1.15]{antonini2016bivariant}). We will now show the following fact:

\begin{thm}\label{mainthm}Let $[D]\in KK_{1}^{G}(A,\mathbb{C})$ and $[M_{l}]\in KK_{1}^{G}(\mathbb{C},\mathbb{C})$ be classes defined by two spectral triples as in Assumptions \ref{assumptions2} and let $[\widehat{D}]\in KK(A\rtimes_{\alpha}G, \mathbb{C})$ be the class defined by the spectral triple \eqref{25bis}. Then 
\begin{equation}\label{Kproduct}
	[\widehat{D}] =  J_{G}[D]\otimes_{C^{*}(G)} I_{G}[M_{l}]. 
\end{equation}

\end{thm}

The proof that $\widehat{D}$ represents the Kasparov product \eqref{Kproduct} relies on the straightforward check of the sufficient conditions provided by the  Kucerovsky criterion \cite{kucerovsky1997kk}, which we  briefly recall in a version for odd ungraded cycles. 
 
\begin{thm}[Theorem $7.2$ in \cite{kaad2013spflow}, cf. Theorem 13 in \cite{kucerovsky1997kk}]\label{Kuce}
	Let $(X,D_{1})$ and $(Y,D_{2})$ be odd ungraded unbounded cycles for $(A,B)$ and $(B,C)$ respectively. Let $E\coloneqq (X\otimes_{B}Y)\otimes \mathbb{C}^{2}$ with grading $\gamma = \left(\begin{matrix}
		1 & 0 \\ 
		0 & -1
	\end{matrix}\right)$ and denote $\sigma_{1} = \left(\begin{matrix}
	0 & -i \\
	i & 0
\end{matrix}\right)$, $\sigma_{2} = \left(\begin{matrix}
0 & 1 \\
1 & 0
\end{matrix}\right)$. Let $T_{x}\colon Y\otimes \mathbb{C}^{2} \rightarrow E$ be the creation operator
\begin{displaymath}
	T_{x}(y_{1},y_{2}) \coloneqq (x\otimes y_{1}, x\otimes y_{2})
\end{displaymath}
for $x\in X$ and $y_{1},y_{2}\in Y$. Assume that $\overline{D}$ is an odd operator such that $(E,\overline{D})$ is an even unbounded cycle for $(A,C)$ such that:
\begin{enumerate}
	\item for all $x$ in a dense submodule of $X$, the graded commutators
	\begin{displaymath}
		\left[\left(\begin{matrix}
			\overline{D} & 0 \\
			0 & D_{2}\sigma_{1}
		\end{matrix}\right),\left(\begin{matrix}
		0 & T_{x} \\
		T_{x}^{*} & 0
	\end{matrix}\right)\right] \colon \Dom \overline{D} \oplus ((\Dom D_{2})\otimes \mathbb{C}^{2}) \rightarrow E\oplus (Y\otimes \mathbb{C}^{2})
	\end{displaymath}
extend to  bounded operators.
\item $\Dom \overline{D}\subseteq \Dom (D_{1}\otimes 1)\oplus \Dom(D_{1}\otimes 1)$
\item There exists $K\geq 0$ such that
\begin{displaymath}
	\langle (D_{1}\otimes 1)\sigma_{2}\xi, \overline{D}\xi\rangle + \langle \overline{D}\xi,  (D_{1}\otimes 1)\sigma_{2}\xi \rangle \geq -K\langle \xi,\xi\rangle 
\end{displaymath}
for all $\xi \in \Dom \overline{D}$.
\end{enumerate}
Then $(E,\overline{D})$ represents the internal Kasparov product of $(X,D_{1})$ and $(Y,D_{2})$. 
\end{thm}
 
 We are now ready to prove our result.
 \begin{proof}[Proof of  Theorem \ref{mainthm}]First of all, let us describe explicitely the image of the equivariant spectral triple $(\mathcal{A},H,D,u)$ under the Kasparov descent (we take for granted the definitions and conventions in Appendix \ref{appA}). From formula \eqref{app1}, consider $B=\mathbb{C}$ and  $C_{c}(G)$ acting on the right on $C_{c}(G,H)$ by right multiplication:
 	\begin{displaymath}
 		(\xi\otimes \delta_{g})\triangleleft \delta_{h} \coloneqq \xi\otimes \delta_{gh}
 	\end{displaymath}
 	for $\xi\in H$ and $\delta_{g}, \delta_{h}\in C_{c}(G)$. The completion of $C_{c}(G,H)$ with the $C^{*}(G)$-valued scalar product 
 	\begin{displaymath}
 		\langle \xi\otimes \delta_{g}, \mu\otimes \delta_{h}\rangle \coloneqq \langle \xi,\mu\rangle\delta_{g^{-1}h},
 	\end{displaymath}as introduced in formula \eqref{app2},
 	defines the right Hilbert $C^{*}(G)$-Hilbert module $H\rtimes G\simeq H\otimes C^{*}(G)$. Using formula \eqref{app3}, we see that the representation $\pi\colon A\rightarrow \mathcal{L}(H)$ induces a representation $\psi$ of  $A\rtimes G$ on $H\otimes C^{*}(G)$  by 
 	\begin{displaymath}
 		\psi(a\delta_{g})(\xi\otimes \delta_{h}) = \pi(a)u_{g}\xi\otimes \delta_{gh}.
 	\end{displaymath}
 	The image of $(\mathcal{A},H,D,u)$ under $J_{G}$ is then just $(C_{c}(G,A), H\otimes C^{*}(G), D\otimes 1,\psi)$.
 	
 	To compare the operator $\widehat{D}$ on $H\otimes \ell^{2}(G)\otimes \mathbb{C}^{2}$ with the operators $D_{1} = D\otimes 1$ on $X=H\otimes C^*(G)$ and $D_{2} = M_{l}$ on $Y=\ell^{2}(G)$, we use the unitary transformation
 	\begin{displaymath}
 		\Phi\colon H\otimes C^{*}(G)\otimes_{C^{*}(G)}\ell^{2}(G) \longrightarrow H\otimes \ell^{2}(G), \quad \Phi(\xi\otimes\delta_{g}\otimes_{C^{*}(G)} \delta_{h}) = \xi\otimes\delta_{gh}
 	\end{displaymath}
 	which is clearly well defined and its inverse is given by $\xi\otimes\delta_{g}\mapsto \xi\otimes\delta_{e}\otimes_{C^{*}(G)} \delta_{g}$. If we still denote by $\Phi$ the doubled map $\Phi\oplus \Phi$ with a slight abuse of notation, we check the conditions of Theorem \ref{Kuce} for $\overline{D}=\Phi^{-1}\widehat{D}\Phi$ as follows:
 	\begin{enumerate}
 		\item Let us fix  $X_{1} \coloneqq (D_{1}-i)^{-1} H\otimes C_{c}(G)$ as the dense subset of $X$ and take $x = \xi\otimes \delta_{g}\in X_{1}$. The boundedness of the commutator of the criterion is equivalent to the boundedness of the following operators:
 		\begin{displaymath}
 			\begin{dcases}
 				\overline{D}T_{x} - T_{x}D_{2}\sigma_{1} & \\
 				D_{2}\sigma_{1}T_{x}^{*} -  T_{x}^{*}\overline{D} & 
 			\end{dcases}
 		\end{displaymath}
 	Let us check the first condition, being the other similar. Given $y_{1},y_{2}\in Y$, we have that $(\overline{D}T_{x} - T_{x}D_{2}\sigma_{1})(y_{1},y_{2})$ is equal to
 	\begin{displaymath}
 		\left(\begin{matrix}
 			D_{1}x\otimes y_{2} -i\Phi^{-1}(1\otimes M_{l})\Phi(x\otimes y_{2}) + ix\otimes M_{l}y_{2} \\
 			D_{1}x\otimes y_{1} +i\Phi^{-1}(1\otimes M_{l})\Phi(x\otimes y_{1}) - ix\otimes M_{l}y_{1} \\
 		\end{matrix}\right).
 	\end{displaymath}
 Since $\Phi^{-1}(1\otimes M_{l})\Phi(x\otimes y_{2}) = x\otimes M_{l}y_{2} + x\otimes \lambda_{g}^{*}[M_{l}, \lambda_{g}]y_{2}$, the previous vector becomes
 	\begin{displaymath}
 	\left(\begin{matrix}
 		D_{1}x\otimes y_{2} -ix\otimes \lambda_{g}^{*}[M_{l}, \lambda_{g}]y_{2} \\
 		D_{1}x\otimes y_{1} +ix\otimes \lambda_{g}^{*}[M_{l}, \lambda_{g}]y_{1}\\
 	\end{matrix}\right)
 \end{displaymath}
which is clearly bounded.
 
 		\item This is immediate since $\Dom \widehat{D} = (\Dom (D\otimes 1)\sigma_{2})\cap \Dom((1\otimes M_{l})\sigma_{1})$. 
 		\item Since the operators $(D\otimes 1)$ and $(1\otimes M_{l})$ commute on $H\otimes \ell^{2}(G)$,  they also commute as operators on $X\otimes_{C^{*}(G)}Y$ after the transformation with $\Phi$. In particular, since $\Phi^{-1}(D\otimes 1) \Phi = D_{1}\otimes 1$, we have that:
 		\begin{displaymath}
 				\langle (D_{1}\otimes 1)\sigma_{2}\xi, \overline{D}\xi\rangle + \langle \overline{D}\xi,  (D_{1}\otimes 1)\sigma_{2}\xi \rangle = 2\langle (D_{1}\otimes 1)\sigma_{2}\xi, (D_{1}\otimes 1)\sigma_{2}\xi\rangle \geq 0.
 		\end{displaymath}
 	The third condition of Theorem \ref{Kuce} is then fulfilled for $K=0$.
 	\end{enumerate}
 	
 \end{proof}

The spectral triple $(C_{c}(G,\mathcal{A}), \widehat{H}, \widehat{D},\hat{\pi}_{2}\rtimes \hat{\Gamma})$ on $A\rtimes_{\alpha,r}G$ as defined in \eqref{25bis} and subsequent discussions has been constructed starting from an equivariant spectral triple: as we are going to show, this triple is  also equivariant  with respect to the dual coaction of $G$. 
 This is not a surprise as it is known that the Kasparov descent map is just the Baaj-Skandalis isomorphism \cite{baaj1989c,blackadar1998k} composed with the functor with forgets the equivariance.  But first, we need to recall what  coactions of $ C^{*}_{r}(G)$ and $ C^{*}_{r}(G)$-comodules are (which, to simplify terminology, are in the following called
respectively coactions of $ G$ and of $ G$-comodules): 

\begin{defn}[cf.\cite{timmermann2008invitation}]\label{77}
	We say that $G$ coacts on a  unital $C^{*}$-algebra $B$ if there exists a unital $C^{*}$-homomorphism (called a \emph{coaction}) $\theta \colon B\rightarrow B\otimes C^{*}_{r}(G)$ such that:
	\begin{enumerate}
		\item $(\theta\otimes \textup{id}) \theta = (\textup{id}_{B}\otimes \Delta)\theta$,
		\item $\linspan\Set{\theta(a)(1_{B}\otimes b) | \, a\in B, b\in C^{*}_{r}(G)}$ is norm dense in $B\otimes C^{*}_{r}(G) $,
	\end{enumerate}
	where the map $\Delta\colon C^{*}_{r}(G)\rightarrow C^{*}_{r}(G)\otimes C^{*}_{r}(G)$  is given by $\Delta(\delta_{g}) = \delta_{g}\otimes \delta_{g}$. In this case we say that $B$ is a \emph{right $G$-comodule}.  For the coaction we adopt an analogue of the Sweedler notation for the coproduct and we denote an element $\theta(b)= \sum_{i=1}^{n} b_{i}\otimes c_{i}$ with $b_{i}\in B$ and $c_{i}\in C^{*}_{r}(G)$ just by 
	\begin{equation}\label{51}
		\theta(b) = \sum b_{(-1)}\otimes b_{(0)},
	\end{equation}
omitting the summation index. 
\end{defn}

\begin{exa}[Dual Coaction]\label{79}
	Consider a $C^{*}$-dynamical system $(A, G,\alpha)$   with $A$ unital and (for simplicity) assume that $G$ is discrete, 
	and set $B=A\rtimes_{\alpha,r}G$. 
	The maps $ i_{A} \colon A\rightarrow 
	B\otimes C^{*}_{r}(G)
	$ and $i_{G}\colon G\rightarrow B\otimes C^{*}_{r}(G)$ given by 
	\begin{equation}\label{74}
		\begin{dcases}
			i_{A}(a) = a\delta_{e}\otimes \delta_{e} &  \\
			i_{G}(g) = 1_{A}\delta_{g}\otimes \delta_{g}, & \\
		\end{dcases}
	\end{equation}
	form a covariant representation of $(A,G,\alpha)$ on 
	$B\otimes C^{*}_{r}(G)$, i.e.,
	\begin{displaymath}
		i_{G}(g)	i_{A}(a)i_{G}(g)^{*} = \delta_{g}a\delta_{g}^{*}\otimes \delta_{g}\delta_{e}\delta_{g}^{*} = \alpha_{g}(a)\delta_{e}\otimes \delta_{e} = i_{A}(\alpha_{g}(a)).
	\end{displaymath}
	The integrated form $\widehat{\alpha}\coloneqq i_{A}\rtimes i_{G}\colon B\rightarrow B\otimes C^{*}_{r}(G)$,
	$a\delta_{g}\mapsto a\delta_{g}\otimes \delta_{g}$, 
	is a coaction (known as the \emph{dual coaction}) of $G$ on $A\rtimes_{\alpha,r}G$ since
	\begin{displaymath}
		(\widehat{\alpha}\otimes \textup{id})\circ \widehat{\alpha}(a\delta_{g}) = a\delta_{g}\otimes \delta_{g}\otimes \delta_{g} = (\textup{id}\otimes \Delta)\circ \widehat{\alpha}(a\delta_{g})
	\end{displaymath}
	and the density condition is trivially satisfied. When $G$ is  abelian, it is known that any coaction $\delta \colon B\rightarrow B\otimes C^{*}_{r}(G)$ is equivalent
	via Fourier transform to an action of the Pontryagin dual group $\widehat{G}$. In particular, 
	the dual coaction $\widehat{\alpha}= i_{A}\rtimes i_{G}$ 
	corresponds to  the \emph{dual action} 
	$\widetilde{\alpha}$ of $\widehat{G}$ on $A\rtimes_{\alpha,r}G$ given by 
	\begin{displaymath}
		\widetilde{\alpha}_{\gamma}(a\delta_{g}) = \overline{\gamma(g)}a\delta_{g}, \qquad \gamma\in \widehat{G}.
	\end{displaymath}
	\demo
\end{exa}

\begin{defn}[cf.\cite{timmermann2008invitation}]\label{76}
	A (unitary) \emph{corepresentation}  of $G$  (or more properly of $C^{*}_{r}(G)$) on a  Hilbert space $H$ is a linear map $\Theta \colon H\rightarrow H\otimes C^{*}_{r}(G)$ such that:
	\begin{enumerate}
		\item $(\Theta\otimes \textup{id}) \Theta = (\textup{id}_{H}\otimes \Delta)\Theta$
		\item $\Theta(H)C^{*}_{r}(G)$ is linearly dense in $H\otimes C^{*}_{r}(G)$
		\item 	$\langle \Theta(x)\, |\, \Theta(y) \rangle = \langle x,y\rangle \cdot 1_{C^{*}_{r}(G)}$ for all $x,y\in H$, 	where $\langle \cdot \,|\, \cdot \rangle$ is the usual scalar product on the external tensor product of the Hilbert modules $H_{\mathbb{C}}$ and $C^{*}_{r}(G)_{C^{*}_{r}(G)}$.   
	\end{enumerate}
\end{defn}

It turns out that when dealing with spectral triples, it is more convenient to see the corepresentations of $G$ on a Hilbert space $H$ as unitary operators $X\in \mathcal{L}(H\otimes C^{*}_{r}(G))$ such that 
\begin{equation}\label{corep}
	(\textup{id}\otimes \Delta)(X)=	X_{(12)}X_{(13)}.
\end{equation}
As shown in \cite[Prop 5.2.2]{timmermann2008invitation}, these two notions coincide in the following sense:
\begin{enumerate}
	\item If $\Theta\colon H\rightarrow H\otimes C^{*}_{r}(G)$ is a unitary corepresentation, then the map 
	\begin{equation}\label{80}
		X\colon H\odot C^{*}_{r}(G) \rightarrow H\otimes C^{*}_{r}(G), \qquad x\otimes q\mapsto \Theta(x)q
	\end{equation}
	extends to a unitary operator $X\in \mathcal{L}(H\otimes C^{*}_{r}(G))$ satisfying \eqref{corep}.
	\item If a unitary $X\in \mathcal{L}(H\otimes C^{*}_{r}(G))$ satisfies \eqref{corep}, then the map
	\begin{displaymath}
		\Theta\colon H\rightarrow H\otimes C^{*}_{r}(G), \qquad x\mapsto X(x\otimes 1_{C^{*}_{r}(G)})
	\end{displaymath}
	is a unitary corepresentation as in Definition \ref{76}. 
\end{enumerate}

The next definition combines the previous two.	

\begin{defn}[cf.\cite{bhowmick2011quantum}]\label{78}
	Let $(\mathcal{B}, H, D, \rchi)$ be an (even or odd) spectral triple on a $G$-comodule unital $C^{*}$-algebra $B$. We say that the triple is \emph{equivariant} for coactions of $G$ if there exists a dense subspace $W\subseteq H$ and a unitary corepresentation $\Theta\colon H \rightarrow H\otimes C^{*}_{r}(G)$ for which:
	\begin{enumerate}
		\item $W$ is a $C^{*}_{r}(G)$-equivariant $B$-module in the sense that for every $b\in B$ and $x\in H$,
		\begin{displaymath}
			\Theta(b\triangleright x) = b_{(-1)}\triangleright x_{(-1)}\otimes b_{(0)}x_{(0)},
		\end{displaymath} 
		\item the operatorial form $X$ of $\Theta$ as in \eqref{80} commutes with $D\otimes 1_{C^{*}_{r}(G)}$ and $\rchi\otimes 1_{C^{*}_{r}(G)}$
		\item $(\textup{id}\otimes \varphi)\textup{Ad}_{U}(b)\in B^{\prime\prime}$ for every $b\in B$ and every state $\varphi$ on $C^{*}_{r}(G)$.  
	\end{enumerate}
\end{defn}

\begin{prop}\label{35}
	The spectral triple $(C_{c}(G,\mathcal{A}), \widehat{H}, \widehat{D},\hat{\pi}_{2}\rtimes \hat{\Gamma})$ is equivariant with respect to the dual coaction $\widehat{\alpha}$ of Example \ref{79}.  	
\end{prop}

\begin{rmk}\label{52}
	If the group $G$ is abelian, Proposition \ref{35} means that the triple on $A\rtimes_{\alpha,r}G$ is $\widehat{G}$-invariant under the unitary representation $V\coloneqq v\oplus v\colon \widehat{G} \rightarrow \mathcal{L}(\widehat{H})$ where $v\colon \widehat{G} \rightarrow \mathcal{L}(H\otimes \ell^{2}(G))$  is given by
	\begin{displaymath}
		  	v_{\rchi}(\xi\otimes \delta_{g}) = \overline{\rchi(g)}\xi\otimes \delta_{g}
	\end{displaymath} 
for $\xi\in H$, $g\in G$ and $\rchi\in \widehat{G}$.
\quad$\diamond$	
	%	\demo
\end{rmk}

\begin{proof} Consider $C^{*}_{r}(G)$ coacting on itself via the map $\Delta(\delta_{g})=\delta_{g}\otimes \delta_{g}$ and the unitary corepresentation  map $\Theta\colon H\otimes \ell^{2}(G)\rightarrow H\otimes \ell^{2}(G)\otimes C^{*}_{r}(G)$ given by 
	\begin{displaymath}
		\Theta(\xi\otimes \delta_{g}) = \xi\otimes \delta_{g}\otimes \delta_{g}, \qquad \xi\in H, \,g\in G.
	\end{displaymath}  	
	According to \eqref{80}, $\Theta$ is equivalent to the unitary operator $U\in \mathcal{L}(H\otimes \ell^{2}(G)\otimes C^{*}_{r}(G))$  given by 
\begin{displaymath}
	U(\xi\otimes \delta_{x}\otimes \delta_{g}) = \xi\otimes \delta_{x}\otimes \delta_{xg}, \qquad x,g\in G
\end{displaymath}
In this way $H\otimes \ell^{2}(G)$ becomes a $C^{*}_{r}(G)$-equivariant $C^{*}_{r}(G)$-module. Consider in fact $b=a\delta_{g}\in C_{c}(G,\mathcal{A})$ and $x=\xi\otimes \delta_{h}\in H\otimes C^{*}_{r}(G)$, then by definition $b_{(-1)}=a\delta_{g}$, $b_{(0)}=\delta_{g}$, $x_{(-1)}=\xi\otimes \delta_{h}$ and $x_{(0)}=\delta_{h}$. Next,
\begin{displaymath}
\begin{split}
		\Theta(b\triangleright x) &=	\Theta(\pi(a)u_{g}\xi\otimes \delta_{gx}) = \pi(a)u_{g}\xi\otimes \delta_{gx}\otimes \delta_{gx} \\
		& = b_{(-1)}\triangleright x_{(-1)}\otimes b_{(0)}x_{(0)}.
\end{split} 
\end{displaymath}
Moreover, it is easy to check that $[D\otimes 1\otimes 1, U]=0$ and $[1\otimes M_{l}\otimes 1, U]=0$ so that $ [\widehat{D}, U\oplus U] =0$.	Finally, we have that $U\hat{\pi}_{2}(a)\hat{\Gamma}_{h}U^{*}(\xi\otimes \delta_{x}\otimes \delta_{g}) = \pi(a)u_{h}\xi\otimes \delta_{hx}\otimes \delta_{hg}$ for any $x,g, h\in G$ and so
	\begin{displaymath}
		(\textup{id}\otimes \varphi)\textup{Ad}_{U}(a\delta_{h}) = \varphi(\delta_{hg})a\delta_{hx}\in C_{c}(G,\mathcal{A})\subseteq (A\rtimes_{r}G)^{\prime\prime}
	\end{displaymath}
	for any state $\varphi$ on $C^{*}_{r}(G)$.	
\end{proof}

We close this section by observing that 
what we have described so far extends to  
a construction of an equivariant odd spectral triple on the crossed product starting from an equivariant \emph{even} spectral triple, and this is unitarily equivalent to the construction described in Remark \ref{34} (assuming $D$ is 
$G$-invariant). Moreover, the aformentioned equivariant triple represents the equivariant Kasparov product of two equivariant triples under the Green-Julg isomorphism 
(for suitable grades)
and is invariant under the coaction of $G$.

\section{The Existence of a Real Structure}\label{J}In this section we construct a real structure on the spectral triple on a crossed product as constructed in Section \ref{sec3.2}, and present sufficient conditions for the first and  second order conditions. The idea is to employ   the tensor product $J_{1}\otimes J_{2}$ of two real structures on two spectral triples $(\mathcal{A}_1, H_{1}, D_{1})$ and 
$(\mathcal{A}_2, H_{2}, D_{2})$ which defines a real structure on the tensor product spectral triple such that the resultant KO-dimension is the sum of the two initial KO-dimensions (with some minor modifications in the case of a grading), cf.  \cite{dkabrowski2011product}. We will check that this construction remains valid also in the case of a crossed product extension. We discuss two cases depending on how the real structure $J$ on the triple $(\mathcal{A},H,D,u)$ interacts with the representation $u\colon G\rightarrow \mathcal{L}(H)$.

\subsection{First Case ($J$ unitarily invariant)} \label{J1} \indent

Let $G$ be a discrete group endowed with a proper Dirac weight $l\colon G\rightarrow \mathbb{R}$ and $(\mathcal{A},H, D,u)$ a $G$-invariant (even or odd) spectral triple on a unital $C^{*}$-algebra $A$ endowed with a real structure $J$ which is \emph{unitarily invariant}, i.e.,
\begin{equation}\label{uJu*}
u_{g}J u_{g}^{*} = J
\end{equation}
for every $g\in G$.   This is for example the case when the triple discussed in Example \ref{nctorus} is endowed with the antilinear operator $J_{1}$ on $L^{2}(S^{1})$ given by the complex conjugation (which gives a real structure on the triple of KO-dimension $1$). We state now our first main result:
\begin{thm}\label{thm1}Suppose $(\mathcal{A}, H,D,J)$ has KO-dimension $n\in \mathbb{Z}_{8}$. If $l\colon G\rightarrow \mathbb{R}$   satisfies $l(g^{-1})=-l(g)$ for all $g\in G$, then the equivariant spectral triple 
$(C_{c}(G,\mathcal{A}),\widehat{H}, \widehat{D}, \hat{\pi}_{2}\rtimes\hat{\Gamma})$ on $A\rtimes_{\alpha,r}G$ admits a real structure $\widehat{J}$ of KO-dimension $n+1$. 
\end{thm}

In view of Proposition \ref{30}, 
a similar result which
provides a real structure of KO-dimension $n-1$ holds for weights such that $l(g^{-1})=l(g)$ for any $g\in G$. However, if we want also the first order condition, Remark \ref{72bis} and Proposition \ref{72} force $G$ to be finite. As this situation is K-homologically trivial, we will not discuss this case.

\begin{rmk}Applying Theorem \ref{thm1} to the triple in Example \ref{nctorus} with real structure $J_{1}$, one recovers precisely the real structure on the noncommutative $2$-torus as described for instance in \cite[Chapter 12.3]{gracia2013elements}.
	
\end{rmk}

The proof of Theorem \ref{thm1} is constructive and relies on the following auxiliary map. 

\begin{lemma}\label{lemma1}
Let  $j\colon H\otimes \ell^{2}(G)\rightarrow H\otimes \ell^{2}(G)$ be the antilinear map defined by
	\begin{equation}\label{34a}
	j(\xi\otimes \delta_{g}) = u_{g}^{*}J\xi\otimes J_{G}\delta_{g} = u_{g}^{*}J\xi\otimes \delta_{g^{-1}},
	\end{equation}
where 
%$C\colon \mathbb{C}\rightarrow \mathbb{C}$ is the usual complex conjugation and 
$J_{G}$ is given by \eqref{29}. Then: 
	\begin{enumerate}
		\item $j$ is  isometric.	
		\item If $J^{2}=\varepsilon$ then $j^{2}=\varepsilon$. 
		\item $j$ maps $A\rtimes_{\alpha}G$ into its commutant.
		\item If $DJ=\varepsilon^{\prime}JD$, then $(D\otimes 1) j=\varepsilon^{\prime}j (D\otimes 1)$.
		\item If $l$ satisfies $l(g^{-1})=-l(g)$ for any $g\in G$, then  $(1\otimes M_{l})j = - j(1\otimes M_{l})$.
	\end{enumerate}
\end{lemma}  

\begin{proof}Point $(1)$ is clear as both $u$ and $J$ are isometric. To prove point $(2)$ note that
	\begin{displaymath}
	j^{2}(\xi\otimes \delta_{g}) = u_{g^{-1}}^{*}Ju_{g}^{*}J\xi\otimes \delta_{g} = u_{g}u_{g}^{*}J^{2}\xi\otimes \delta_{g} = \varepsilon (\xi\otimes \delta_{g}).
	\end{displaymath} 
	Point $(3)$ comes by a straightforward computation: for any $a,b\in A$ and $g,h\in G$ we have
	\begin{displaymath}
	\begin{split}
	\left[\hat{\pi}_{2}(a)\hat{\Gamma}_{h}, j\hat{\pi}_{2}(b)\hat{\Gamma}_{g}j^{-1}\right](\xi\otimes \delta_{x}) &=	\pi(a)u_{h}u_{gx^{-1}}^{*}J\pi(b)u_{g}J^{-1}u_{x}^{*}\xi\otimes \delta_{hxg^{-1}} \\
    &\qquad \qquad - u_{gx^{-1}h^{-1}}^{*}J\pi(b)u_{g}J^{-1}u_{hx}^{*}\pi(a)u_{h}\xi \otimes \delta_{hxg^{-1}}\\
	&= \left[\pi(a), J\pi(\alpha_{hxg^{-1}}(b))J^{-1}\right]u_{h}\xi \otimes \delta_{hxg^{-1}} \\
	&=0
	\end{split}
	\end{displaymath}
	since $J$ satisfies the zeroth order condition. To prove point $(4)$ note that
	\begin{displaymath}
	(D\otimes 1) j (\xi\otimes \delta_{g})= Du_{g}^{*}J\xi\otimes \delta_{g^{-1}} =  \varepsilon^{\prime} u_{g}^{*}JD \xi\otimes \delta_{g^{-1}}= \varepsilon^{\prime} j(D\otimes 1)(\xi\otimes \delta_{g})
	\end{displaymath}
	by the invariance of $D$. Point $(5)$ is clear. 
\end{proof}

\noindent 
We then claim that the equivariant real structure $\widehat{J}$  for $\widehat{D}$ when $n$ is odd is given by  
\begin{itemize}
	\item $		\widehat{J}= j\otimes cc \quad$  for $n=3,7$
	
	\item $\widehat{J}= j\otimes cc\circ \sigma_{2} \quad  $ for $n=1,5$
\end{itemize}
on the Hilbert space $\widehat{H}=H\otimes \ell^{2}G\otimes \mathbb{C}^{2}$ (here $cc$ denotes the complex conjugation operator). When $n$ is even, the equivariant real structure $\widehat{J}$  for $\widehat{D}$ is instead given by
\begin{itemize}
	\item $\widehat{J}=\rchi J\otimes J_{G} \quad $ for $n=0,4$
	\item $\widehat{J}=J\otimes J_{G} \quad $ for $n=2,6$
\end{itemize}
on the Hilbert space $\widehat{H} = H\otimes \ell^{2}(G)$, where $J_{G}$ is the flip morphism defined in \eqref{29}.

\begin{proof}[Proof of Theorem \ref{thm1}]Suppose as a first case that the triple $(\mathcal{A}, H, D)$ is odd.  The zeroth-order condition directly comes from the zeroth order condition in Lemma \ref{lemma2} and the peculiar diagonal/anti-diagonal form of $\widehat{J}$. Let us now discuss the triple of signs $(\varepsilon,\varepsilon^{\prime},\varepsilon^{\prime} )$: using the hypothesis on $l$, $j$ anti-commutes with $M_{l}$ on $\ell^{2}(G)$ and so the real structure $\widehat{J}$ has the same $\varepsilon$  sign as $J$ when $n=3,7$ and the opposite when $n=1,5$. Analogously, the $\varepsilon^{\prime}$ sign remains the same for $n= 3,7$ and changes for $n=1,5$ (that is, is always $+1$). Further $\widehat{J}$ is always even with respect to the grading $\rchi=\sigma_{3}$ when $n=3,7$ and odd when $n=1,5$. We have therefore checked the theorem for all the possible cases of $(\mathcal{A}, H, D)$ odd. 
	
	Suppose now that $(\mathcal{A}, H, D,J)$ is an even  real triple with respect to the grading
	$\rchi= \sigma_{3}$, $H= H_{0}\oplus H_{1}$ and 
	\begin{displaymath}
	D=\left(\begin{matrix}
	0 & D_{0} \\
	D_{0}^{*} & 0
	\end{matrix}\right).
	\end{displaymath}  
	Recall that $\widehat{D}$ is given by  \eqref{dual1}. By the compatibility conditions of $J$ with the grading $\rchi$ we deduce that $J$ must be of the following form:
	\begin{itemize}
		\item $J=\left(\begin{matrix}
		j_{1} & 0 \\
		0 & j_{2}
		\end{matrix}\right)$ for $n=0,4$,
		\item $J= \left(\begin{matrix}
		0 & j_{1} \\
		j_{2} & 0
		\end{matrix}\right)$ for $n=2,6$.
	\end{itemize}
	So now we have just to check case by case: if $n=0,4$ then by assumption
	\begin{displaymath}
	\left(\begin{matrix}
	0 & D_{0} \\
	D^{*}_{0} & 0
	\end{matrix}\right)\left(\begin{matrix}
	j_{1} & 0 \\
	0 & j_{2}
	\end{matrix}\right) =\left(\begin{matrix}
	j_{1} & 0 \\
	0 & j_{2}
	\end{matrix}\right)\left(\begin{matrix}
	0 & D_{0} \\
	D^{*}_{0} & 0
	\end{matrix}\right).
	\end{displaymath}
	Using this equation it is easy to check that 
	\begin{displaymath}
	\left(\begin{matrix}
	1\otimes M_{l} & D_{0}\otimes 1 \\
	D^{*}_{0}\otimes 1 & -1\otimes M_{l}
	\end{matrix}\right)\left(\begin{matrix}
	j_{1}\otimes J_{G} & 0 \\
	0 & -j_{2}\otimes J_{G}
	\end{matrix}\right) = - \left(\begin{matrix}
	j_{1}\otimes J_{G} & 0 \\
	0 & -j_{2}\otimes J_{G}
	\end{matrix}\right)\left(\begin{matrix}
	1\otimes M_{l} & D_{0}\otimes 1 \\
	D^{*}_{0}\otimes 1 & -1\otimes M_{l}
	\end{matrix}\right)
	\end{displaymath}
	so that the two triples have different  $\varepsilon^{\prime}$ signs. The $\varepsilon$ sign remains the same and this proves that we get dimensions $1$ and $5$ respectively. In an analogous way one shows that the dimensions $2$ and $6$ go  to the dimensions $3$ and $7$ respectively. 
\end{proof}

\begin{rmk}\label{85}
The assumption that $[D,u_{g}]=0$ for any $g\in G$ is essentially necessary in order to prove Theorem \ref{thm1}. Indeed, suppose there is a real structure $J$ on $(\mathcal{A},H,D)$, so that  $DJ=\varepsilon^{\prime}JD$ by definition. An essential step to prove Theorem \ref{thm1} is Lemma \ref{lemma1}$(4)$, namely the fact that 
		\begin{equation}\label{84}
	Du_{g}J=\varepsilon^{\prime}u_{g}JD, \qquad \forall g\in G.
	\end{equation}
These two conditions together imply that $D$ must be $G$-invariant: indeed, using \eqref{84} we see that $Du_{g}=\varepsilon^{\prime}u_{g}JDJ^{-1}$ and so
$$
		[D,u_{g}] = \varepsilon^{\prime}u_{g}JDJ^{-1} - u_{g}DJJ^{-1} = u_{g}\left(\varepsilon^{\prime}JD -DJ\right)J^{-1} =0.
$$
Note that this computation is independent of the fact that $J$ is unitarily equivalent, which is instead an assumption needed to show that $\widehat{J}$ satisfies the zeroth order condition. \demo
\end{rmk}

The following results show that the crossed product spectral triple construction is compatible with the first and second order conditions. 

\begin{prop}Let $G$ be a discrete group endowed with a proper group homomorphism $l\colon G\rightarrow \mathbb{R}$ and let $(\mathcal{A},H,D,u)$ be a $G$-invariant (even or odd) spectral triple on a unital $C^{*}$-algebra $A$ endowed with a unitarily invariant real structure $J$. 
If $(\mathcal{A},H,D,J)$ satisfies the first order condition, then the spectral triple 
$(C_{c}(G,\mathcal{A}), \widehat{H}, \widehat{D}, \widehat{J})$ on $A\rtimes_{\alpha,r}G$ also satisfies the first order condition.
\end{prop}

\begin{proof}We prove the first order condition for $(\mathcal{A},H,D,u,J)$ odd; the even case is similar. For any $a,b\in A$ and $g,h\in G$, the desired commutator
	\begin{displaymath}
	\left[\left[D\otimes 1 \pm i\otimes M_{l}, \hat{\pi}_{2}\rtimes\hat{\Gamma}(a\delta_{g})\right], j\hat{\pi}_{2}\rtimes \hat{\Gamma} (b\delta_{h})j^{-1}\right] 
	\end{displaymath}
	is equal to the sum of the following two pieces:
	\begin{displaymath}
	C_{1} =  \left[\left[D\otimes 1 ,\hat{\pi}_{2}(a)\hat{\Gamma}_{g}\right],j\hat{\pi}_{2}(b)\hat{\Gamma}_{h}j^{-1}\right] 
	\end{displaymath} 
	\begin{displaymath}
	C_{2}= \pm i\left[\left[ 1\otimes M_{l},\hat{\pi}_{2}(a)\hat{\Gamma}_{g}\right],j\hat{\pi}_{2}(b)\hat{\Gamma}_{h}j^{-1}\right] .
	\end{displaymath}
	We will separately prove that they are vanishing. On the one hand, using the invariance of $D$ with respect to the action of $G$, we have that 
	\begin{displaymath}
	\begin{split}
	C_{1}(\xi\otimes \delta_{x}) &= \left[D\otimes 1, \hat{\pi}_{2}(a)\hat{\Gamma}_{g}\right]\left(u_{hx^{-1}}^{*}J\pi(b)u_{h}J^{-1}u_{x}^{*}\xi\otimes \delta_{xh^{-1}}\right) \\
	& \qquad \qquad -j\hat{\pi}_{2}(b)\hat{\Gamma}_{h}j^{-1}\left(D\pi(a)u_{g}\xi\otimes \delta_{gx} - \pi(a)u_{g}D\xi\otimes \delta_{gx}\right)\\
	&= [D,\pi(a)]u_{gxh^{-1}}J\pi(b)u_{h}J^{-1}u_{x}^{*}u_{g}^{*}u_{g}\xi\otimes \delta_{gxh^{-1}} \\
	& \qquad \qquad - u_{gxh^{-1}}J\pi(b)u_{h}J^{-1}u_{gx}^{*}[D,\pi(a)]u_{g}\xi\otimes \delta_{gxh^{-1}} \\
	&= \left[\left[D,\pi(a)\right],Ju_{gxh^{-1}}\pi(b)u_{gxh^{-1}}^{*}J^{-1}\right]u_{g}\xi\otimes \delta_{gxh^{-1}} \\
	&= 0
	\end{split}
	\end{displaymath}
	by the first order condition for $J$. On the other hand
	\begin{displaymath}
	\begin{split}
	\pm iC_{2}(\xi\otimes \delta_{x}) 	&= \left[1\otimes M_{l},\hat{\pi}_{2}(a)\hat{\Gamma}_{g}\right]\left(u_{hx^{-1}}^{*}J\pi(b)u_{h}J^{-1}u_{x}^{*}\xi\otimes \delta_{xh^{-1}}\right) \\
	& \qquad \qquad - j\hat{\pi}_{2}(b)\hat{\Gamma}_{h}j^{-1}\left(\pi(a)u_{g}\xi\otimes l(gx)\delta_{gx} - \pi(a)u_{g}\xi \otimes l(x)\delta_{gx}\right)	\\
	& = \left(l(gxh^{-1}) - l(xh^{-1})\right)\pi(a)u_{gxh^{-1}}J\pi(b)u_{gxh^{-1}}^{*}J^{-1}u_{g}\xi\otimes \delta_{gxh^{-1}} \\
	& \qquad \qquad - \left(l(gx) - l(x)\right)u_{gxh^{-1}}J\pi(b)u_{gxh^{-1}}^{*}J^{-1}\pi(a)u_{g}\xi\otimes \delta_{gxh^{-1}}.
	\end{split}
	\end{displaymath}
	As $l\colon G\rightarrow \mathbb{R}$ is a homomorphism, we have that $	l(gxh^{-1})- l(xh^{-1}) = l(gx) - l(x)$. Then
	\begin{displaymath}
	\pm iC_{2}(\xi\otimes \delta_{x}) = \left(l(gx) - l(x)\right)\left[\pi(a),J\pi(\alpha_{gxh^{-1}}(b))J^{-1}\right]u_{g}\xi\otimes \delta_{gxh^{-1}}
	\end{displaymath}
	and this is zeroth as $J$ implements the zeroth order condition. 
\end{proof}

\begin{prop}Let $G$ be a discrete group endowed with a proper group homomorphism $l\colon G\rightarrow \mathbb{R}$ and let $(\mathcal{A},H,D,u)$ be a $G$-invariant (even or odd) spectral triple on a unital $C^{*}$-algebra $A$ endowed with a unitarily invariant real structure $J$ which satisfies the first order condition. If $(\mathcal{A},H,D,J)$ satisfies the second order condition, then the spectral triple 
$(C_{c}(G,\mathcal{A}), \widehat{H}, \widehat{D}, \widehat{J})$ on $A\rtimes_{\alpha,r}G$ also satisfies the second order condition. 
\end{prop}

\begin{proof}	
	Let us  focus on $(\mathcal{A},H,D,u,J)$ odd as the even case is similar. With a slight abuse of notation, let us denote $\widehat{D}=D\otimes 1 \pm i\otimes M_{l}$. To prove the required commutation relation 
	\begin{displaymath}
	[\widehat{D},\hat{\pi}_{2}(a)\hat{\Gamma}_{h}]j[\widehat{D},\hat{\pi}_{2}(b)\hat{\Gamma}_{g}]j^{-1} = j[\widehat{D},\hat{\pi}_{2}(b)\hat{\Gamma}_{g}]j^{-1}[\widehat{D},\hat{\pi}_{2}(a)\hat{\Gamma}_{h}]
	\end{displaymath}
	we will show that the following four commutators are vanishing:
	\begin{displaymath}
	\begin{split}
	C_{1} &= \left[[D\otimes 1,\hat{\pi}_{2}(a)\hat{\Gamma}_{h}], j[D\otimes 1,\hat{\pi}_{2}(b)\hat{\Gamma}_{g}]j^{-1}\right]\\
		C_{2} &= \left[[D\otimes 1,\hat{\pi}_{2}(a)\hat{\Gamma}_{h}], j[\pm i\otimes M_{l},\hat{\pi}_{2}(b)\hat{\Gamma}_{g}]j^{-1}\right]\\
			C_{3} &= \left[[\pm i\otimes M_{l},\hat{\pi}_{2}(a)\hat{\Gamma}_{h}], j[D\otimes 1,\hat{\pi}_{2}(b)\hat{\Gamma}_{g}]j^{-1}\right]\\
				C_{4} &= \left[[\pm i\otimes M_{l},\hat{\pi}_{2}(a)\hat{\Gamma}_{h}], j[\pm i\otimes M_{l},\hat{\pi}_{2}(b)\hat{\Gamma}_{g}]j^{-1}\right]\\
	\end{split}
	\end{displaymath}
	for any $a,b\in A$ and $g,h\in G$. 
	First, note that:
	\begin{displaymath}
	\begin{split}
	j[D\otimes 1,\hat{\pi}_{2}(b)\hat{\Gamma}_{g}]j^{-1}(\xi\otimes\delta_x)&=u_{gx^{-1}}^{*}J[D,\pi(b)]u_gJ^{-1}u_x^{*}\xi\otimes\delta_{xg^{-1}}\nonumber\\
	&=J[D,\pi(\alpha_{xg^{-1}}(b))]J^{-1}\xi\otimes\delta_{xg^{-1}}
	\end{split}
	\end{displaymath}
 and that
		\begin{displaymath}
	\begin{split}
	j[\pm i\otimes M_{l},\hat{\pi}_{2}(b)\hat{\Gamma}_{g}]j^{-1}(\xi\otimes\delta_x)&=\mp iu_{gx^{-1}}^{*}J\pi(b)u_gJ^{-1}u_x^{*}\xi\otimes l(g)\delta_{xg^{-1}}\nonumber\\
	&=\mp iJ\pi(\alpha_{xg^{-1}}(b))J^{-1}\xi\otimes l(g)\delta_{xg^{-1}}
	\end{split}
	\end{displaymath}
	as $l$ is a homomorphism. It is then easy to see that 
	\begin{displaymath}
	C_{1} = \left[[D,\pi(a)], J[D,\pi(\alpha_{hxg^{-1}}(b))]J^{-1}\right]u_h\xi\otimes\delta_{hxg^{-1}}
	\end{displaymath}
	which vanishes since $J$ implements the second order condition. Next, since $l(h)l(g)=l(g)l(h)$ for any $g,h\in G$, we have that
	\begin{displaymath}
	C_{4}= \left[\pi(a), J\pi(\alpha_{hxg^{-1}}(b))J^{-1} \right]u_h\xi\otimes l(g)l(h)\delta_{hxg^{-1}}
	\end{displaymath}
	vanishes by the zeroth order condition for $J$. 
	Furthermore, the two mixed terms
	\begin{displaymath}
\begin{split}
	C_{2} &= i\left[[D,\pi(a)], J\pi(\alpha_{hxg^{-1}}(b))J^{-1}\right]u_h\xi\otimes l(g)\delta_{hxg^{-1}}\\
C_{3} &= \pm i\left[\pi(a), J[D,\pi(\alpha_{hxg^{-1}}(b))]J^{-1}\right] u_h\xi\otimes l(h)\delta_{hxg^{-1}}
\end{split}
	\end{displaymath}
	vanish for the first order condition for $J$. 
The peculiar diagonal/anti-diagonal form of $\widehat{J}$ then brings the thesis. 	

\end{proof}

\subsection{Second Case ($J$ twisted invariant)}\label{35b}\indent

Let $G$ be discrete group endowed with a proper Dirac weight $l\colon G\rightarrow \mathbb{R}$ and $(\mathcal{A},H, D,u)$ a $G$-invariant (even or odd) spectral triple on a unital $C^{*}$-algebra $A$ endowed with a real structure $J$ which is \emph{twisted invariant}, namely such that
\begin{equation}\label{uJu}
	u_{g}J u_{g} = J
\end{equation}
for every $g\in G$.  This is for example the case when the triple discussed in Example \ref{nctorus} is endowed with the antilinear operator $J_{2}$ on $L^{2}(S^{1})$ given by $J_{2}f(t) = \overline{f(-t)}$ (which gives a real structure on the triple of KO-dimension $7$). We will prove the following fact.

\begin{thm} \label{thm2}Suppose $(\mathcal{A}, H,D,J)$ has KO-dimension $n\in \mathbb{Z}_{8}$. If $G$ is abelian, the  $\widehat{G}$-invariant spectral triple 
$(C_{c}(G,\mathcal{A}),\widehat{H}, \widehat{D}, \hat{\pi}_{2}\rtimes\hat{\Gamma})$ on $A\rtimes_{\alpha,r}G$ admits a  real structure $\widetilde{J}$ of KO-dimension $n-1$. 
\end{thm}

%\begin{rmk}
%	\nrt{
%		The assumption that $G$ is abelian might seem unnecessarily restrictive (and thus of low interest)
%		compared to the discussion in Subsection \ref{J1}, where $G$ need not be commutative. However, this rigidity is just apparent. Indeed, when we focus our attention to spectral triples which are non-degenerate  (a condition that is necessary for example to have a compact quantum metric space structure) we discover that the hypothesis  that $l\colon G\rightarrow \mathbb{R}$ is a homomorphism as in Theorem \ref{thm1} forces the commutativity of $G$ (cf. Example \ref{exa}). 
%	}
%\end{rmk}

The proof of Theorem \ref{thm2} is constructive and relies on the properties of the following auxiliary map. 

\begin{lemma}\label{lemma2}
Let $G$ be an abelian discrete group and 
$j\colon H\otimes \ell^{2}(G)\rightarrow H\otimes \ell^{2}(G)$ an antilinear map defined by
	\begin{equation}\label{34b}
	j(\xi\otimes \delta_{g}) = u_{g}J\xi\otimes cc \,\delta_{g}.
	\end{equation}
%	where 
% $C\colon \mathbb{C}\rightarrow \mathbb{C}$ is the usual complex conjugation. 
Then: 
	\begin{enumerate}
		\item $j$ is isometric
		\item If $J^{2}=\varepsilon$ then $j^{2}=\varepsilon$. 
		\item $j$ maps $A\rtimes_{\alpha}G$ into its commutant.
		\item If $DJ=\varepsilon^{\prime}JD$, then $(D\otimes 1) j=\varepsilon^{\prime}j (D\otimes 1)$
		\item $(\pm i\otimes M_{l})j = - j(\pm i\otimes M_{l})$
	\end{enumerate}
\end{lemma}  

\begin{proof}Point $(1)$ is clear as both $u$ and $J$ are isometric. To prove point $(2)$ note that
	\begin{displaymath}
	j^{2}(\xi\otimes \delta_{g}) = u_{g}Ju_{g}J\xi\otimes \delta_{g} = u_{g}u_{g}^{*}J^{2}\xi\otimes \delta_{g} = \varepsilon (\xi\otimes \delta_{g}).
	\end{displaymath} 
	Point $(3)$ is a straightforward computation: for any $a,b\in A$ and $g,h\in G$ we have
		\begin{displaymath}
	\begin{split}
	\left[\hat{\pi}_{2}(a)\hat{\Gamma}_{h}, j\hat{\pi}_{2}(b)\hat{\Gamma}_{g}j^{-1}\right](\xi\otimes \delta_{x}) &= \pi(a)u_{h}u_{xg}J\pi(b)u_{g}J^{-1}u_{x}^{*}\xi\otimes \delta_{xgh}\\
	& \qquad \qquad - u_{xhg}J\pi(b)u_{g}J^{-1}u_{xh}^{*}\pi(a)u_{h}\xi \otimes \delta_{xhg}\\
	&= \left[\pi(a), J\pi(\alpha_{xgh}^{-1}(b))J^{-1}\right]u_{h}\xi \otimes \delta_{xgh} =0\\
	\end{split}
	\end{displaymath}
	as $J$ satisfies the zero order condition and $G$ is abelian. To prove point $(4)$ note that
	\begin{displaymath}
	(D\otimes 1) j (\xi\otimes \delta_{g})= Du_{g}J\xi\otimes \delta_{g} =  \varepsilon^{\prime} u_{g}JD \xi\otimes \delta_{g}= \varepsilon^{\prime} j(D\otimes 1)(\xi\otimes \delta_{g})
	\end{displaymath}
	by the invariance of $D$. Point $(5)$ is clear as $J$ is anti-linear. 
\end{proof}

\noindent 
We then claim that the equivariant real structure $\widetilde{J}$  for $\widehat{D}$ when $n$ is odd is given by  
	\begin{itemize}
	\item $		\widetilde{J}= j\otimes cc\circ \sigma_{1}\quad $  for $n=3,7$
	
	\item $\widetilde{J}= j\otimes cc\circ \sigma_{3}\quad  $ for $n=1,5$
\end{itemize}
on the Hilbert space $\widehat{H}=H\otimes \ell^{2}G\otimes \mathbb{C}^{2}$. When $n$ is even,  $\widetilde{J}$   is instead given by
\begin{itemize}
		\item $\widetilde{J}=J\otimes \textup{cc}\quad $ for $n=0,4$
	\item $\widetilde{J}= \rchi J\otimes \textup{cc}\quad $ for $n=2,6$
\end{itemize}
on the Hilbert space $\widehat{H} = H\otimes \ell^{2}(G)$.

\begin{proof}[Proof of Theorem \ref{thm2}]Suppose as a first case that the triple $(\mathcal{A}, H, D)$ is odd. The zeroth-order condition  comes directly from the zeroth order condition in Lemma \ref{lemma2} and the  diagonal/anti-diagonal form of $\widetilde{J}$. Let us now discuss the triple of signs $(\varepsilon,\varepsilon^{\prime},\varepsilon^{\prime\prime} )$: from the previous lemma we easily deduce that the real structure $\widetilde{J}$ has the same $\varepsilon$  sign as $J$. An easy computation shows that the $\varepsilon^{\prime}$ sign remains the same for $n= 3,7$ and changes for $n=1,5$ (that is, it is always $+1$). Further $\widetilde{J}$ is always odd with respect to the grading $\rchi=\sigma_{3}$ when $n=3,7$ and even when $n=1,5$. We have therefore checked the theorem for all the possible cases of $(\mathcal{A}, H, D)$ odd. \\
 
 Suppose now that the real triple $(\mathcal{A}, H, D,J)$ is even with respect to the grading $\rchi$ and suppose that $\rchi= \sigma_{3}$, $H= H_{0}\oplus H_{1}$ and 
 \begin{displaymath}
D=\left(\begin{matrix}
0 & D_{0} \\
D_{0}^{*} & 0
\end{matrix}\right).
 \end{displaymath} 
By the compatibility conditions of $J$ with the grading $\rchi$ we deduce that $J$ must be of the following form:
\begin{itemize}
	\item $J=\left(\begin{matrix}
	j_{1} & 0 \\
	0 & j_{2}
	\end{matrix}\right)$ for $n=0,4$,
	\item $J= \left(\begin{matrix}
	0 & j_{1} \\
	j_{2} & 0
	\end{matrix}\right)$ for $n=2,6$.
\end{itemize}
Now we check case by case. If $n=0,4$ then by assumption
\begin{displaymath}
\left(\begin{matrix}
0 & D_{0} \\
D^{*}_{0} & 0
\end{matrix}\right)\left(\begin{matrix}
j_{1} & 0 \\
0 & j_{2}
\end{matrix}\right) =\left(\begin{matrix}
j_{1} & 0 \\
0 & j_{2}
\end{matrix}\right)\left(\begin{matrix}
0 & D_{0} \\
D^{*}_{0} & 0
\end{matrix}\right)
\end{displaymath}
and, recalling that $\widehat{D}$ is given by  \eqref{dual1}, we get 
\begin{displaymath}
\left(\begin{matrix}
1\otimes M_{l} & D_{0}\otimes 1 \\
D^{*}_{0}\otimes 1 & -1\otimes M_{l}
\end{matrix}\right)\left(\begin{matrix}
j_{1}\otimes 1 & 0 \\
0 & j_{2}\otimes 1
\end{matrix}\right) =\left(\begin{matrix}
j_{1}\otimes 1 & 0 \\
0 & j_{2}\otimes 1
\end{matrix}\right)\left(\begin{matrix}
1\otimes M_{l} & D_{0}\otimes 1 \\
D^{*}_{0}\otimes 1 & -1\otimes M_{l}
\end{matrix}\right).
\end{displaymath}
Thus the triples on $A$ and on $A\rtimes_{\alpha,r}G$ have the same sign $\varepsilon^{\prime}$. 
Also the sign $\varepsilon$ remains the same, but since 
the grading disapears, the dimension $0$ (or $8\mod 8$) goes to $7$ and the dimension $4$ goes to $3$. In an analogous way one shows that dimensions $2$ and $6$ go  to  dimensions $1$ and $5$ respectively.
\end{proof}

The following results show that the crossed product spectral triple construction is compatible with the first and second order conditions. 

\begin{prop}Let $G$ be an abelian discrete group endowed with a proper first-order Dirac weight $l\colon G\rightarrow \mathbb{R}$ and let $(\mathcal{A},H,D,u)$ be a $G$-invariant (even or odd) spectral triple on a unital $C^{*}$-algebra $A$ endowed with a twisted invariant real structure $J$. If $(\mathcal{A},H,D,J)$ satisfies the first order condition, then $(C_{c}(G,\mathcal{A}), \widehat{H}, \widehat{D}, \widetilde{J})$ also satisfies the first order condition. 
\end{prop}

\begin{proof}We prove the first order condition only for 
$(\mathcal{A},H,D,u,J)$ odd; the even case is similar. For any $a,b\in \mathcal{A}$ and $g,h\in G$, the desired commutator
	\begin{displaymath}
	\left[\left[D\otimes 1 \pm i\otimes M_{l}, \hat{\pi}_{2}\rtimes\hat{\Gamma}(a\delta_{g})\right], j\hat{\pi}_{2}\rtimes \hat{\Gamma} (b\delta_{h})j^{-1}\right] 
	\end{displaymath}
	is equal to the sum of the following two pieces:
	\begin{displaymath}
	C_{1} =  \left[\left[D\otimes 1 ,\hat{\pi}_{2}(a)\hat{\Gamma}_{g}\right],j\hat{\pi}_{2}(b)\hat{\Gamma}_{h}j^{-1}\right] 
	\end{displaymath} 
	\begin{displaymath}
	C_{2}= \pm i\left[\left[ 1\otimes M_{l},\hat{\pi}_{2}(a)\hat{\Gamma}_{g}\right],j\hat{\pi}_{2}(b)\hat{\Gamma}_{h}j^{-1}\right] .
	\end{displaymath}
	We prove that they are separately vanishing. On the one hand, using the invariance of $D$ with respect to the action of $G$, we have that 
	\begin{displaymath}
	\begin{split}
	C_{1}(\xi\otimes \delta_{x}) &= \left[D\otimes 1, \hat{\pi}_{2}(a)\hat{\Gamma}_{g}\right]\left(u_{hx}J\pi(b)u_{h}J^{-1}u_{x}^{*}\xi\otimes \delta_{hx}\right) \\
	& \qquad\qquad   - j\hat{\pi}(b)\hat{\Gamma}_{h}j^{-1}\left(D\pi(a)u_{g}\xi\otimes \delta_{gx} - \pi(a)u_{g}D\xi \otimes \delta_{gx}\right)\\
	&= [D,\pi(a)]u_{xgh}J\pi(b)u_{h}J^{-1}u_{x}^{*}\xi\otimes \delta_{xgh} \\
	& \qquad \qquad - u_{xgh}J\pi(b)u_{h}J^{-1}u_{xg}^{*}[D,\pi(a)]u_{g}\xi\otimes \delta_{xgh} \\
	&= \left[\left[D,\pi(a)\right],Ju_{xgh}^{*}\pi(b)u_{xgh}J^{-1}\right]u_{g}\xi\otimes \delta_{xgh} \\
	&= 0
	\end{split}
	\end{displaymath}
	since $J$ implements the first order condition. On the other hand
	\begin{displaymath}
	\begin{split}
	\pm iC_{2}(\xi\otimes \delta_{x}) 	&= \left[1\otimes M_{l},\hat{\pi}_{2}(a)\hat{\Gamma}_{g}\right]u_{xh}J\pi(b)u_{h}J^{-1}u_{x}^{*}\xi\otimes \delta_{xgh} \\
	& \qquad \qquad - j\hat{\pi}_{2}(b)\hat{\Gamma}_{h}j^{-1}\left(\pi(a)u_{g}\xi\otimes l(xg)\delta_{xg} - \pi(a)u_{g}\xi \otimes l(x)\delta_{xg}\right)	\\
	& = \left(l(xgh) - l(xh)\right)\pi(a)u_{xgh}J\pi(b)u_{h}J^{-1}u_{x}^{*}\xi\otimes \delta_{xgh} \\
	& \qquad \qquad - \left(l(xg) - l(x)\right)u_{xgh}J\pi(b)u_{h}J^{-1}u_{xg}^{*}\pi(a)u_{g}\xi\otimes \delta_{xgh}.
	\end{split}
	\end{displaymath}
	Since $l\colon G\rightarrow \mathbb{R}$ is of first order, we have that $	l(xgh)- l(xh) = l(xg) - l(x)$. Then
	\begin{displaymath}
	\pm iC_{2}(\xi\otimes \delta_{x}) = \left(l(xg) - l(x)\right)\left[\pi(a),J\pi(\alpha_{xgh}^{-1}(b))J^{-1}\right]u_{g}\xi\otimes \delta_{xgh}
	\end{displaymath}
	which is zero since $J$ implements the zeroth order condition. 
\end{proof}

\begin{prop}Let $G$ be an abelian discrete group endowed with a proper first-order Dirac weight $l\colon G\rightarrow \mathbb{R}$ and let $(\mathcal{A},H,D,u)$ be a $G$-invariant (even or odd) spectral triple on a unital $C^{*}$-algebra $A$ endowed with a twisted invariant real structure $J$ which satisfies the first order condition. If $(\mathcal{A},H,D,J)$ satisfies the second order condition, then $(C_{c}(G,\mathcal{A}), \widehat{H}, \widehat{D}, \widetilde{J})$ also satisfies the second order condition. 
\end{prop}

\begin{proof}	
	Let us  focus on $(\mathcal{A},H,D,u,J)$ odd cause the even case is similar. With a slight abuse of notation, let us denote $\widehat{D}=D\otimes 1 \pm i\otimes M_{l}$. To prove the required commutation relation 
	\begin{displaymath}
	[\widehat{D},\hat{\pi}_{2}(a)\hat{\Gamma}_{h}]j[\widehat{D},\hat{\pi}_{2}(b)\hat{\Gamma}_{g}]j^{-1} = j[\widehat{D},\hat{\pi}_{2}(b)\hat{\Gamma}_{g}]j^{-1}[\widehat{D},\hat{\pi}_{2}(a)\hat{\Gamma}_{h}]
	\end{displaymath}we will prove that the following four commutators are vanishing:
	\begin{displaymath}
	\begin{split}
	C_{1} &= \left[[D\otimes 1,\hat{\pi}_{2}(a)\hat{\Gamma}_{h}], j[D\otimes 1,\hat{\pi}_{2}(b)\hat{\Gamma}_{g}]j^{-1}\right]\\
	C_{2} &= \left[[D\otimes 1,\hat{\pi}_{2}(a)\hat{\Gamma}_{h}], j[\pm i\otimes M_{l},\hat{\pi}_{2}(b)\hat{\Gamma}_{g}]j^{-1}\right]\\
	C_{3} &= \left[[\pm i\otimes M_{l},\hat{\pi}_{2}(a)\hat{\Gamma}_{h}], j[D\otimes 1,\hat{\pi}_{2}(b)\hat{\Gamma}_{g}]j^{-1}\right]\\
	C_{4} &= \left[[\pm i\otimes M_{l},\hat{\pi}_{2}(a)\hat{\Gamma}_{h}], j[\pm i\otimes M_{l},\hat{\pi}_{2}(b)\hat{\Gamma}_{g}]j^{-1}\right]\\
	\end{split}
	\end{displaymath}
	for any $a,b\in \mathcal{A}$ and $g,h\in G$. The  diagonal/anti-diagonal form of $\widetilde{J}$ then brings the thesis. First of all note that:
	\begin{displaymath}
	\begin{split}
	j[D\otimes 1,\hat{\pi}_{2}(b)\hat{\Gamma}_{g}]j^{-1}(\xi\otimes\delta_x)&=u_{gx}J[D,\pi(b)]u_gJ^{-1}u_x^{*}\xi\otimes\delta_{gx}\\
	&=J[D,\pi(\alpha_{gx}^{-1}(b))]J^{-1}\xi\otimes\delta_{gx}
	\end{split}
	\end{displaymath}
	 and that
	\begin{displaymath}
	\begin{split}
	j[\pm i\otimes M_{l},\hat{\pi}_{2}(b)\hat{\Gamma}_{g}]j^{-1}(\xi\otimes\delta_x)	&=\mp iu_{gx}J\pi(b)u_gJ^{-1}u_x^{*}\xi\otimes l(g)\delta_{gx}\nonumber\\
	&=\mp iJ\alpha_{gx}^{-1}(b)J^{-1}\xi\otimes l(g)\delta_{gx}\label{1C6}
	\end{split}
	\end{displaymath}
	as $l$ is of first order. It is relatively easy then to compute
	\begin{displaymath}
	C_{1} = \left[[D,\pi(a)], J[D,\pi(\alpha_{hgx}^{-1}(b))]J^{-1}\right]u_h\xi\otimes\delta_{hgx}
	\end{displaymath}
	that vanishes since $J$ implements the second order condition. Furthermore, as $l(h)l(g)=l(g)l(h)$ for any $g,h\in G$, we have that
	\begin{displaymath}
	C_{4}= \left[\pi(a), J\pi(\alpha_{hgx}^{-1}(b))J^{-1} \right]u_h\xi\otimes l(g)l(h)\delta_{hgx}
	\end{displaymath}
	vanishes since $J$ implements the zeroth order condition. Finally, the two mixed terms
	\begin{displaymath}
	\begin{split}
	C_{2} &= \mp i\left[[D,\pi(a)], J\pi(\alpha_{hgx}^{-1}(b))J^{-1} \right]u_h\xi\otimes l(g)\delta_{hgx} \\
	C_{3} &= \pm i\left[ \pi(a), J[D,\pi(\alpha_{hgx}^{-1}(b))]J^{-1}\right] u_h\xi\otimes l(h)\delta_{hgx}
	\end{split}
	\end{displaymath}
	vanish by the first order condition for $J$.

\end{proof}

\subsection{Equivariant Real Structures}\indent

In the previous two subsections we constructed real structures $\widehat{J}$ and  (for abelian $G$) $\widetilde{J}$ on 
$(C_{c}(G,\mathcal{A}),\widehat{H},\widehat{D})$ starting from a real structure $J$ on $(\mathcal{A},H,D,u)$ which is suitably $G$-invariant.
In this subsection we give a unifying picture 
interpreting the relations between $J$ with $u$ in terms of
the (unitary) action of the Hopf $*$-algebra $\mathbb{C}G$ endowed with a suitable $*$-structure. 
This will explain the reason why in the case of $\widetilde{J}$
we must assume the group is abelian. Furthermore, we will show that in both cases $\widehat{J}$ and $\widetilde{J}$ are equivariant under the dual coaction of $\mathbb{C}G$ consistently with $J$. But first, we need to recall some basic facts and definitions about Hopf algebras.\\

	Let $\mathcal{H}$ be a unital Hopf $*$-algebra, with coproduct homomorphism $\Delta \colon \mathcal{H}\rightarrow \mathcal{H}\otimes \mathcal{H}$, counit homomorphism $\varepsilon \colon \mathcal{H}\rightarrow \mathbb{C}$, and antipode anti-homomorphism $S\colon \mathcal{H}\rightarrow \mathcal{H}$, satisfying the usual axioms (c.f \cite{timmermann2008invitation}) and an antilinear involutive anti-homomorphism  $*$ such that 
	\begin{enumerate}
		\item $\Delta(h^{*}) = h_{(1)}^{*}\otimes h_{(2)}^{*}$ for every $h\in \mathcal{H}$,
		\item $\varepsilon(h^{*}) = \overline{\varepsilon(h)}$ for every $h\in \mathcal{H}$,
		\item $(S\circ *)^{2} = \textup{id}$.
	\end{enumerate}
We adopt Sweedler's notation $\Delta h = \sum h_{(1)}\otimes  h_{(2)}$ for $h\in \mathcal{H}$ with the summation symbol  often omitted for the sake of brevity. We are mostly  interested in the following example.

\begin{exa}\label{57}
	Let $G$ be a group and $\mathcal{H}=\mathbb{C}G$ its group algebra. It is a Hopf algebra with respect to the maps 
	$\Delta$, $\varepsilon$ and $S$ determined on generators by
	\begin{displaymath}
		\Delta(\delta_{g}) = \delta_{g}\otimes \delta_{g}, \qquad \varepsilon(\delta_{g}) = 1, \qquad S\delta_{g}=\delta_{g^{-1}}
	\end{displaymath}
	and extended linearly. 	This Hopf algebra admits a canonical $*$-structure given by the \emph{anti}-linear extension of the map
	\begin{equation}\label{691}
		\delta_{g}^{*} = \delta_{g^{-1}}.
	\end{equation}
	However, it is not unique 	when $G$ is abelian:
 the \emph{anti}-linear extension of the map 
	\begin{equation}\label{692}
		\delta_{g}^{\star}=\delta_{g}
	\end{equation} is also a $*$-structure. (Note that we use two different star symbols).	
	\demo
\end{exa}

The closure of the Hopf algebra $\mathcal{H}=\mathbb{C}G$
is the quantum group $C^{*}_{r}(G)$, which we used in Definition\,\ref{77} and Example\,\ref{79} together with its coactions. Below we will  instead need the {\em actions} of 
$\mathcal{H}$.

Given a Hopf algebra $\mathcal{H}$, we say that an algebra $A$ is a \emph{left $\mathcal{H}$-module algebra} if $A$ is a left $\mathcal{H}$-module and the representation is compatible with  the algebra structure of $A$, namely
\begin{displaymath}
	h\triangleright (a_{1}a_{2}) = (h_{(1)}\triangleright a_{1})(h_{(2)}\triangleright a_{2})
\end{displaymath}
for any $h\in \mathcal{H}$ and $a_{1}, a_{2}\in A$. If $A$ is unital, we further require that 
\begin{displaymath}
	h\triangleright 1 = \varepsilon(h) 
\end{displaymath}
for any $ h\in\mathcal{H}$. Let $A$ be a left $H$-module algebra and $M$ a left $A$-module. We say that $M$ is a \emph{left $H$-equivariant $A$-module} if $M$ is a left $H$-module and
\begin{displaymath}
	h\triangleright (am) = (h_{(1)}\triangleright a)(h_{(2)}\triangleright m)
\end{displaymath}
for any $h\in H$, $a\in A$ and $m\in M$. In the following, when dealing with a Hopf $*$-algebra $\mathcal{H}$ and an $\mathcal{H}$-module algebra $A$ endowed with a $*$-involution, we will always assume that the action of $\mathcal{H}$ is compatible with the star structure of $A$ in the sense that
\begin{equation}\label{86}
	(h\triangleright a)^{*} = (Sh)^{*}\triangleright a^{*},
	\qquad \forall a\in A,\; h\in \mathcal{H}.
\end{equation}

\begin{defn}[cf.\cite{sitarz2003equivariant}]
	Let $\mathcal{H}$ be a Hopf algebra and $(\mathcal{A}, H,D)$ be a spectral triple over an $\mathcal{H}$-module algebra $A$. We say that the triple is $\mathcal{H}$-\emph{equivariant} if there exists a dense subspace $W\subseteq H$ for which:
	\begin{enumerate}
		\item  $W$ is an $\mathcal{H}$-equivariant $\mathcal{A}$-module, that is:
		\begin{displaymath}
			h\triangleright (\pi(a)v) = \pi(h_{(1)}\triangleright a)(h_{(2)}\triangleright v)
		\end{displaymath}
		for any $h \in \mathcal{H}$, $v\in W$ and $a\in \mathcal{A}$. 
		\item the commutator $[D,h\,\triangleright]$ is bounded on its domain for any $h \in \mathcal{H}$. 
	\end{enumerate}
	If the commutators $[D,h\,\triangleright]$ vanish for every $h \in \mathcal{H}$ we say that the triple is \emph{invariant}.
\end{defn}

\begin{defn}[cf.\cite{sitarz2003equivariant}]\label{70}
	Let $\mathcal{H}$ be a Hopf $*$-algebra and $(\mathcal{A}, H, D )$ an $\mathcal{H}$-equivariant spectral triple over the $\mathcal{H}$-module $C^{*}$-algebra $A$. A real structure $J$ is said to be \emph{equivariant} if there exists a dense subspace $V\subseteq H$ such that for any $h\in \mathcal{H}$
	\begin{equation}\label{38}
		Jh\triangleright J^{-1} = (Sh)^{*}\triangleright 
	\end{equation}
as operators on $V$. 
\end{defn}

We can now explain the commutation relations \eqref{uJu*} and \eqref{uJu} between $J$ and $u_g$ as the 
$\mathcal{H}$-equivariance in the sense of Definition \ref{70} corresponding to the two different $*$-structures \eqref{691} and \eqref{692} respectively on the Hopf algebra $\mathcal{H}=\mathbb{C}G$ as in Example \ref{57}. In the first case we use the obvious actions of 
$\mathbb{C}G$ on $H$ and $A$ 
\begin{displaymath}
	h\triangleright \xi = u_{h}\xi, \qquad g\triangleright a = \alpha_{g}(a)
\end{displaymath}
to make $A$ a $\mathbb{C}G$-module algebra and $H$ a 
$\mathbb{C}G$-equivariant $\mathcal{A}$-module. Then equation \eqref{38} for the $*$-structure \eqref{691} on 
$\mathbb{C}G$ becomes 
\begin{displaymath}
		Ju_{g}J^{-1} = u_{(Sg)^{*}} = u_{g}
	\end{displaymath}
	which means precisely that $J$ is unitarily invariant
\eqref{uJu*}. In the second case (when $G$ is abelian) we use the (less) obvious actions of $\mathbb{C}G$ on $H$ and $A$ 
\begin{displaymath}
	h\triangleright \xi = u_{h}^{*}\xi, \qquad g\triangleright a = \alpha_{g^{-1}}(a)
\end{displaymath}
to make $A$ a $\mathbb{C}G$-module algebra and $H$ a 
$\mathbb{C}G$-equivariant $\mathcal{A}$-module. Then equation \eqref{38} for the $*$-structure \eqref{692} on 
$\mathbb{C}G$ becomes
\begin{displaymath}
	Ju_{g}J^{-1} = u_{(Sg)^{*}} = u_{g^{-1}}= u_{g}^{*}
\end{displaymath} 
which means precisely that $J$ is twisted invariant \eqref{uJu}. Note that in both cases the compatibility condition \eqref{86} holds true.\\

With this unifying picture, we summarize the two constructions of this section in the following table:

\begin{center}
	\begin{tabular}{c|ccc}
		\toprule
		Group $G$ & Discrete & Discrete Abelian& \\
		$*$-structure on $\mathbb{C}G $&  $*\delta_{g} = \delta_{g^{-1}}$ & $\star\delta_{g} = \delta_{g}$ &\\
		%\midrule
		Equivariance of $J$ & $u_{g}Ju_{g}^{*} = J$ & $u_{g}Ju_{g} = J$ &\\
		weight $l\colon G\rightarrow \mathbb{R}$ & homomorphism & constant + homom. &\\
		auxiliary map $j$ & $j(\xi\otimes \delta_{g})= u_{g}^{*}J\xi\otimes \delta_{g^{-1}}$ & $j(\xi\otimes \delta_{g})= u_{g}J\xi\otimes \delta_{g}$ & \\
		& & &\\
		\multirow{4}{*}{}
		& \multirow{4}{*}{$
			\widehat{J} = \begin{dcases*} 
				& \\
				& \\
				& \\
				& \\
			\end{dcases*}$}  \hspace*{-7mm} $j\otimes \textup{cc}$  &\multirow{4}{*}{$
			\widetilde{J} = \begin{dcases*} 
				& \\
				& \\
				& \\
				& \\
			\end{dcases*}$} \hspace*{-6mm} $j\otimes cc\circ \sigma_{1}$ & \multirow{4}{*}{$
			\textup{for}\, n= \begin{dcases*} 
				& \\
				& \\
				& \\
				& \\
			\end{dcases*}$}\hspace*{-6mm} $3,7$\\
		real structure on            & \hspace*{19mm}$j\otimes cc\circ \sigma_{2}$ & \hspace*{11mm}$j\otimes cc\circ \sigma_{3}$ & \hspace*{15mm}	$1,5$\\
		$(C_{c}(G,\mathcal{A}), \widehat{H},\widehat{D})$                & \hspace*{14mm}$\rchi J\otimes J_{G}$ & \hspace*{10mm}$J\otimes \textup{cc}$ & \hspace*{15mm}  $	0,4$ \\
		&\hspace*{12mm} $J\otimes J_{G}$ &\hspace*{12mm} $\rchi J\otimes \textup{cc}$ & \hspace*{15mm}  $2,6 $\\
		& & &\\                                                     
		{KO}-dim     & $n +1 $& $n -1$& \\
		\bottomrule
	\end{tabular}
\end{center}
~\\

Let us now prove that the real structures $\widehat{J}$ and $\widetilde{J}$ are equivariant for coactions of $G$. First, we need a definition. 

	\begin{defn}[cf.\cite{bhowmick2011quantum}]\label{778}
	Let $(\mathcal{B}, H, D, \rchi)$ an (even or odd) spectral triple  equivariant for coaction of $G$ as in 
Definition\,\ref{78} and let $U\in\mathcal{L}(H\otimes C^{*}_{r}(G))$ be the unitary corepresentation of $G$ on $H$. A real structure $J$ on $(\mathcal{B}, H, D, \rchi)$
	 is said to be \emph{equivariant for coactions of $G$} if 
	\begin{equation}
		(J\otimes *) U  = U(J\otimes 1)
	\end{equation}
	on $H\otimes 1_{C^{*}_{r}(G)}$. 
\end{defn}

\begin{prop}
	Let $G$ be a discrete group endowed with a proper group homomorphism $l\colon G\rightarrow \mathbb{R}$ and $(\mathcal{A},H, D,u)$ a $G$-invariant (even or odd) spectral triple on a unital $C^{*}$-algebra $A$ endowed with a unitarily invariant real structure $J$. The real structure $\widehat{J}$ on the equivariant spectral triple $(C_{c}(G,\mathcal{A}),\widehat{H}, \widehat{D}, \hat{\pi}_{2}\rtimes\hat{\Gamma})$ on $A\rtimes_{\alpha,r}G$ defined in Theorem \ref{thm1} is equivariant for the dual coaction of $G$.  
\end{prop}

\begin{rmk}
	If $G$ is abelian and the dual coaction 
	$\widehat{\alpha}$ is Fourier-transformed into the dual action of $\widehat{G}$, one can show that if $J$ is unitarily invariant then $\widehat{J}$ is also unitarily invariant under the action $V$ given in 	Remark \ref{52}. 
\end{rmk}

\begin{proof}
For any $\xi\in H$ and $g\in G$ we have
	\begin{displaymath}
		\begin{split}
			(j\otimes *)U(\xi\otimes \delta_{g}\otimes \delta_{e}) &= (j\otimes *)(\xi\otimes \delta_{g}\otimes \delta_{g}) = 	u_{g}^{*}J\xi\otimes \delta_{g^{-1}}\otimes \delta_{g^{-1}} \\
			&= U(u_{g}^{*}J\xi\otimes \delta_{g^{-1}}\otimes \delta_{e}) \\
			&= U(j(\xi\otimes \delta_{g})\otimes \delta_{e}).
		\end{split}
	\end{displaymath}
	The 
	diagonal/anti-diagonal	form of $\widehat{J}$ leads to the thesis. 
\end{proof}

\begin{prop}	Let $G$ be a discrete abelian group endowed with a proper Dirac weight $l\colon G\rightarrow \mathbb{R}$ and $(\mathcal{A},H, D,u)$ a $G$-invariant (even or odd) spectral triple on a unital $C^{*}$-algebra $A$ endowed with a twisted invariant real structure $J$. The real structure 
$\widetilde{J}$ on the equivariant spectral triple $(C_{c}(G,\mathcal{A}),\widehat{H}, \widehat{D}, \hat{\pi}_{2}\rtimes\hat{\Gamma})$ on $A\rtimes_{\alpha,r}G$ defined in Theorem \ref{thm1} is twisted invariant. 

\end{prop}

\begin{proof} Noting that $jv_{\rchi} = v_{\rchi}^{*}j$ since characters $\rchi\in \widehat{G}$ are complex-valued, the diagonal/anti-diagonal
 form of $\widetilde{J}$ leads to the thesis.
	\end{proof}

\section{The Existence of an Orientation Cycle}\label{c}

The orientability condition generalizes 
a (noncommutative) differential top form
in terms of Hochschild homology, 
which we now briefly recall \cite{loday2013cyclic}. Let $A$ be a complex unital algebra and $M$
an $A$-bimodule.
For every positive $n\in \mathbb{N}$  
define the $A$-module of \emph{Hochschild $n$-chains} 
(with coefficients in $M$) to be 
$C_{n}(M,A)= M\otimes A^{\otimes n}$ and $C_{0}(M,A)= M$. 
The \emph{Hochschild boundary} is the family of maps $b_{n}\colon C_{n}(M,A)\rightarrow C_{n-1}(M,A)$ given on pure elements by 
\begin{equation}\label{65}
\begin{split}
b_{n}\left( m\otimes a_{1}\otimes \cdots \otimes a_{n}\right) &= ma_{1}\otimes a_{2}\otimes \cdots \otimes a_{n} \\
& \qquad + \sum_{i=1}^{n-1}(-1)^{i}m\otimes a_{1}\otimes \cdots\otimes  a_{i-1}\otimes a_{i}a_{i+1}\otimes \cdots \otimes a_{n} \\
& \qquad \qquad + (-1)^{n} a_{n}m \otimes a_{1}\otimes \cdots \otimes  a_{n-1}
\end{split}
\end{equation}
if $n\geq 1$ and $b_0(m)=0$, and extended by linearity. 
It turns out that $\left(C_{\bullet}(M,A), b \right)$ is a chain complex and its homology is the Hochschild homology with coefficients in $M$. Choosing $M=A$ as an $A$-bimodule with the usual left and right multiplication, we get the Hochschild chain complex of $A$.

Let now $(\mathcal{A},H,D, \rchi)$ be an 
even or odd spectral triple on $A$,
and let $J$ be a real structure of KO-dimension $n\in\mathbb{Z}_{8}$. For a Hochschild $n$-chain $c = \sum a_{0}\otimes a_{1}\otimes \cdots \otimes a_{n} \in C_{n}(\mathcal{A}, \mathcal{A})$ 
set
\begin{displaymath}
	\pi_{D}(c) \coloneqq \sum \pi(a_{0})[D,\pi(a_{1})]\cdots [D,\pi(a_{n})].
\end{displaymath}
\begin{defn}
	A spectral triple $(\mathcal{A},H,D, \rchi)$ on $A$ is \emph{strongly orientable} if there exists a Hochschild $n$-cycle $c\in C_{n}(\mathcal{A},\mathcal{A})$ such that $\pi_{D}(c)=\rchi$. 
\end{defn}

In noncommutative geometry it is useful to consider also a weaker notion of orientability. Consider the case in which the $A$-module is $M= A\otimes A^{\textup{op}}$, where $A^{\textup{op}}$ denotes the opposite algebra, with the left and right actions of $A$ given on $m\otimes n\in A\otimes A^{\textup{op}}$ by:
\begin{displaymath}
a(m\otimes n)b = amb\otimes n\qquad a,b\in A.
\end{displaymath}
For a Hochschild $n$-chain $c = \sum (a_{0}\otimes b_{0})\otimes a_{1}\otimes \cdots \otimes a_{n}$ in $C_{n}(\mathcal{A}\otimes \mathcal{A}^{\textup{op}},\mathcal{A})$ we define the map
\begin{equation}\label{4534}
\pi_{D}(c) \coloneqq \sum \pi(a_{0})J\pi(b^{*}_{0})J^{-1}[D,\pi(a_{1})]\cdots [D,\pi(a_{n})].
\end{equation}

\begin{defn} A real spectral triple $(\mathcal{A},H,D,J,\rchi)$ on $A$ is \emph{orientable} if there exists a Hochschild $n$-cycle  $c\in C_{n}(\mathcal{A}\otimes \mathcal{A}^{\textup{op}},\mathcal{A})$ such that $\pi_{D}(c)=\rchi$. 
\end{defn}

Note that if $\mathcal{A}$ is unital and $\pi(1_{\mathcal{A}}) = \textup{id}_{H}$, then every strong orientation cycle\\ 
$c = \sum a_{0}\otimes a_{1}\otimes \cdots \otimes a_{n}$ in $C_{n}(\mathcal{A}, \mathcal{A})$ induces a (weak) orientation cycle $c' = \sum (a_{0}\otimes 1_{A})\otimes a_{1}\otimes \cdots \otimes a_{n}$ in $C_{n}(\mathcal{A}\otimes \mathcal{A}^{\textup{op}}, \mathcal{A})$ as
\begin{displaymath}
	\begin{split}
		\pi_{D}(c') &= \sum \pi(a_{0})J\pi(1_{A})J^{-1}[D,\pi(a_{1})]\cdots [D,\pi(a_{n})] \\
		&= \sum \pi(a_{0})[D,\pi(a_{1})]\cdots [D,\pi(a_{n})] \\
		&= 	\pi_{D}(c) = \rchi .
	\end{split}
\end{displaymath}
We now make $G$  act on Hochschild chains.

\begin{defn}\label{83}Let $(A, G,\alpha)$ be a $C^{*}$-dynamical system. For every Hochschild $n$-chain $c=\sum(a_{0}\otimes b_{0})\otimes a_{1}\otimes \cdots\otimes a_{n}\in C_{n}(A\otimes A^{\textup{op}}, A)$, we define 
	\begin{equation}\label{45}
		\alpha_{g}(c) \coloneqq \sum(\alpha_{g}(a_{0})\otimes b_{0})\otimes \alpha_{g}(a_{1})\otimes \cdots \otimes \alpha_{g}( a_{n})
	\end{equation}
and say that $c$ is \emph{$G$-invariant} if $\alpha_{g}(c)=c$ for every $g\in G$. 
\end{defn}

\begin{rmk}
	For any   Hochschild $n$-chain $c$  we have that $	b\alpha_{g}(c) = \alpha_{g}(bc)$ by the definition of the Hochschild boundary. In particular, if $c$ is a cycle then $\alpha_{g}(c)$ is also a cycle.  
	\qquad
\end{rmk}

In equation \eqref{45} the elements $b_{0}$ play no essential role. When dealing with the orientation property, this is reflected in the following fact. 
\begin{lemma}\label{360}
Let $(\mathcal{A},H,D,\rchi,u)$ be a $G$-invariant real spectral triple on $A$ with  a unitarily invariant real structure $J$ and let
 $c = \sum (a_{0}\otimes b_{0})\otimes a_{1}\otimes \cdots \otimes a_{n}$  be a Hochschild cycle in $C_{n}(\mathcal{A}\otimes \mathcal{A}^{\textup{op}}, \mathcal{A})$. 
 Define
	\begin{equation}\label{361}
		c_{g} \coloneqq \sum (a_{0}\otimes \alpha_{g}(b_{0}))\otimes a_{1}\otimes \cdots \otimes a_{n}
	\end{equation}
for any $g\in G$. Then $\pi_{D}(c_{g}) = \textup{Ad} u_{g}\circ \pi_{D}(\alpha_{g}(c))$. 
\end{lemma}

\begin{proof}Using the fact that $[D,u_{g}]=0$ for any $g\in G$ and that $Ju_{g}=u_{g}J$ for any $g\in G$, we have that
	\begin{displaymath}
		\begin{split}
			\pi_{D}(c_{g}) &= \sum \pi(a_{0})J\pi(\alpha_{g}(b_{0})^{*})J^{-1}[D,\pi(a_{1})]\cdots  [D,\pi(a_{n})]\\
			&= u_{g}\sum \pi(\alpha_{g}(a_{0}))J\pi(b_{0}^{*})J^{-1}[D,\pi(\alpha_{g}(a_{1}))]\cdots  [D,\pi(\alpha_{g}(a_{n}))]u_{g}^{*}\\
			&= \textup{Ad} u_{g}\circ \pi_{D}(\alpha_{g}(c)).
		\end{split}
	\end{displaymath}
\end{proof}

We state now our second main result. 

 \begin{thm}\label{59}Let $G$ be a discrete group and $l\colon G\rightarrow \mathbb{R}$ a proper homomorphism. Let $(\mathcal{A}, H, D,u)$ be an (even or odd) $G$-invariant spectral triple on a unital $C^{*}$-algebra $A$ and $J$ a unitarily invariant real structure. Then:
 	\begin{enumerate}
 		\item If $(\mathcal{A}, H, D)$ is orientable and the orientation cycle $c$ is $G$-invariant, then the real spectral triple $(C_{c}(G,\mathcal{A}), \widehat{H}, \widehat{D}, \hat{\pi}_{2}\rtimes \hat{\Gamma}, \widehat{J})$ on $A\rtimes_{\alpha,r} G$ admits an orientation cycle $\hat{c}$.
 		\item If $c$ is a strong orientation cycle, then $\hat{c}$ is also a strong orientation cycle. 
 	\end{enumerate} 
 \end{thm}
As suggested in  \cite[Chapter 6]{zucca2013dirac}, the idea of the proof is to twist the prescription described in \cite{dkabrowski2011product}, where the shuffle product is used to create a cycle on a tensor product spectral triple. 
 \begin{defn}\label{cad}
  For any Hochschild $n$-chain 
 	\begin{displaymath}
 		c = \sum (a_{0}\otimes b_{0})\otimes a_{1}\otimes \cdots \otimes a_{n} \in C_{n}(A\otimes A^{\textup{op}}, A)
 	\end{displaymath}and any $1$-chain	$\delta = \sum(\delta_{g}\otimes \delta_{h})\otimes \delta_{f} \in C_{1}(Q\otimes Q^{\textup{op}},Q)$ for $Q=C^{*}_{r}(G)$, we define their \emph{twisted shuffle product} as the Hochschild $(n+1)$-chain 
in $C_{n+1}(B\otimes B^{\textup{op}}, B)$ for $B=A\otimes C^{*}_{r}(G)$: 
 	\begin{displaymath}
 		\begin{split}
 			c\rtimes_{\alpha}\delta &\coloneqq \sum(a_{0}\delta_{g}\otimes  b_{0}\delta_{h})\otimes \delta_{f} \otimes a_{1} \otimes \cdots \otimes a_{n} \\
 			& \qquad \qquad  + \sum_{j=2}^{n}(-1)^{j-1}\sum(a_{0}\delta_{g}\otimes  b_{0}\delta_{h})\otimes \alpha_{f}(a_{1})\otimes \cdots \otimes \alpha_{f}(a_{j-1})\otimes \delta_{f}\otimes a_{j}\otimes \cdots\otimes a_{n} \\
 			& \qquad \qquad \qquad \qquad + (-1)^{n}\sum(a_{0}\delta_{g}\otimes  b_{0}\delta_{h})\otimes \alpha_{f}(a_{1})\otimes \cdots \otimes \alpha_{f}(a_{n})\otimes \delta_{f}.
 		\end{split}
 	\end{displaymath}
\end{defn}	 	
 	
 	Note that for simplicty we denote by $a$ the element $a\delta_{e}$ and by $\delta_{f}$ the element $1_{A}\delta_{f}$. Under the assumption of covariance, namely that $\delta_{g}a = \alpha_{g}(a)\delta_{g}$ for any $a\in A$ and $g\in G$, the twisted shuffle product also defines a chain over the crossed product $A\rtimes_{\alpha,r}G$. 
	  Note also that if $\alpha=\textup{id}$ then the twisted shuffle product is really the shuffle product of the two chains as defined in \cite[Chapter $4.2$]{loday2013cyclic} (up to a sign depending on the  length of the chain).

\begin{prop}\label{63}
	For any $G$-invariant $c\in C_{n}(A\otimes A^{\textup{op}},A)$ and any $\delta\in C_{1}(Q\otimes Q^{\textup{op}}, Q)$ we have that
	\begin{equation}
	b(c\rtimes_{\alpha}\delta) = bc\rtimes_{\alpha}\delta  + c\rtimes b\delta
	\end{equation} 
as chains over $A\rtimes_{\alpha,r}G$. 
\end{prop}

\begin{proof}
By bilinearity, we can suppose that $c$ and $\delta$ are pure tensors:
\begin{displaymath}
	c =(a_{0}\otimes b_{0})\otimes a_{1}\otimes \cdots \otimes a_{n}, \qquad 	\delta = (\delta_{g}\otimes\delta_{h})\otimes \delta_{f}.
	\end{displaymath} 
Since $b\delta = \delta_{gf}\otimes \delta_{h} - \delta_{fg}\otimes \delta_{h}\in Q\otimes Q^{\textup{op}}$, 
the untwisted shuffle product with $c$ is equal to
\begin{equation}\label{66}
\begin{split}
c\rtimes b\delta &= \left(a_{0}\delta_{gf}\otimes b_{0}\delta_{h}\right)\otimes a_{1}\otimes \cdots\otimes a_{n} \\
&\qquad\qquad   - \left(a_{0}\delta_{fg}\otimes b_{0}\delta_{h}\right)\otimes a_{1}\otimes \cdots\otimes a_{n}.
\end{split}
\end{equation}	
For the sake of simplicity, we set $m=a_{0}\delta_{g}\otimes b_{0}\delta_{h}$ and write $c\rtimes_{\alpha}\delta$ in Definition\,\eqref{cad} as $c\rtimes_{\alpha}\delta = \sum_{j=1}^{n+1}c_{j}$. Let us compute $bc_{j}$ for every $j=1, \dots, n+1$. First,
\begin{equation}\label{auxiliary 1}
\begin{split}
bc_{1} &= \left(a_{0}\delta_{gf}\otimes b_{0}\delta_{h}\right)\otimes a_{1}\otimes \cdots\otimes a_{n} \\
& \qquad -m\otimes \delta_{f}a_{1}\otimes a_{2}\otimes \cdots \otimes a_{n} \\
&\qquad + \sum_{i=1}^{n-1}(-1)^{i+1}m\otimes \delta_{f}\otimes a_{1}\otimes \cdots \otimes a_{i}a_{i+1}\otimes \cdots\otimes a_{n}\\
&\qquad  +(-1)^{n+1}a_{n}m\otimes \delta_{f}\otimes a_{1}\otimes \cdots\otimes a_{n-1}.
\end{split}
\end{equation}
Next, for $j=2,\dots, n$ we have:
\begin{displaymath}
\begin{split}
bc_{j}&= (-1)^{j-1}m\otimes 
\alpha_{f}(a_{1})\otimes \cdots \otimes \delta_{f}\otimes a_{j}\otimes\cdots\otimes a_{n}\\
& \qquad + \sum_{i=1}^{j-2}(-1)^{i+j-1}m\otimes \alpha_{f}(a_{1})\otimes \cdots\otimes \alpha_{f}(a_{i}a_{i+1})\otimes \cdots\\ 
&\hspace{6cm}\cdots \otimes \alpha_{f}(a_{j-1})\otimes \delta_{f}\otimes a_{j}\otimes \cdots\otimes a_{n}\\
&\qquad + (-1)^{j-1}(-1)^{j-1} m\otimes \alpha_{f}(a_{1})\otimes \cdots\otimes \alpha_{f}(a_{j-1})\delta_{f}\otimes a_{j}\otimes \cdots \otimes a_{n} \\
& \qquad +(-1)^{j-1}(-1)^{j} m\otimes \alpha_{f}(a_{1})\otimes \cdots \otimes \alpha_{f}(a_{j-1})\otimes \delta_{f}a_{j}\otimes \cdots\otimes a_{n}\\
& \qquad +\sum_{i=j}^{n-1}(-1)^{i+j}m\otimes \alpha_{f}(a_{1})\otimes \cdots \otimes \alpha_{f}(a_{j-1})\otimes \delta_{f}\otimes a_{j}\otimes \cdots \otimes a_{i}a_{i+1}\otimes \cdots \otimes a_{n}\\
&\qquad +(-1)^{n+1}(-1)^{j-1}a_{n}m\otimes \alpha_{f}(a_{1})\otimes \cdots \otimes \alpha_{f}(a_{j-1})\otimes \delta_{f}\otimes a_{j}\otimes \cdots \otimes a_{n-1}
\end{split}
\end{displaymath}
with the convention that for $j=2$ the first summation is neglected. Finally:
\begin{equation}\label{auxiliary 3}
\begin{split}
bc_{n+1} &= (-1)^{n}m\otimes \alpha_{f}(a_{1})\otimes \alpha_{f}(a_{2})\otimes \cdots\otimes\alpha_{f}(a_{n})\otimes \delta_{f} \\
& \qquad + (-1)^{n}\sum_{i=1}^{n-1}(-1)^{i}m\otimes \alpha_{f}(a_{1})\otimes \cdots \otimes \alpha_{f}(a_{i}a_{i+1})\otimes \cdots \otimes \alpha_{f}(a_{n})\otimes \delta_{f} \\
&\qquad + (-1)^{n}(-1)^{n}m\otimes \alpha_{f}(a_{1})\otimes \cdots \otimes \alpha_{f}(a_{n-1})\otimes \alpha_{f}(a_{n})\delta_{f} \\
&\qquad  + (-1)^{n}(-1)^{n+1}\left(\delta_{f}a_{0}\delta_{g}\otimes b_{0}\delta_{h}\right)\otimes \alpha_{f}(a_{1})\otimes \cdots \otimes \alpha_{f}(a_{n}).
\end{split}
\end{equation}
Since $\delta_{f}a_{0}\delta_{g}=\alpha_{f}(a_{0})\delta_{fg}$ and the cycle $c$ is $G$-invariant,
the last line of \eqref{auxiliary 3} can be rewritten as
\begin{displaymath}
-(a\delta_{fg}\otimes b_{0}\delta_{h})\otimes a_{1}\otimes \cdots \otimes a_{n}
\end{displaymath}
Together with the first line of \eqref{auxiliary 1}, these summands are precisely $c\rtimes b\delta$ (see \eqref{66}). Now we note that the two central lines of every $bc_{j}$ for $j=2, \dots, n$ form a telescopic summation that, together with the second line in \eqref{auxiliary 1} and the third line in \eqref{auxiliary 3}, sum up to zero. What remains is precisely $bc\rtimes_{\alpha}\delta$. 
\end{proof}
We  can now prove our second main theorem.
\begin{proof}[Proof of Theorem \ref{59}]Since the Dirac weight $l$ is proper, there exists $g\in G$ such that  $l(g)\neq 0$. Consider then the Hochschild $1$-cycle
	\begin{displaymath}
			\Delta_{g} \coloneqq (\delta_{g^{-1}}\otimes \delta_{e})\otimes \delta_{g} \in C_{1}(Q\otimes Q^{\textup{op}}, Q).
	\end{displaymath}
If $c$ is the $G$-invariant orientation cycle of the triple 
$(A,H, D, J,u)$, then the twisted shuffle product 
$c\rtimes_{\alpha} \Delta_{g}$ is also a cycle by 
Proposition\,\ref{63}. 
We will show that the normalised shuffle product
	\begin{equation}\label{chat}
	\hat{c}= \frac{1}{M} c\rtimes_{\alpha} \Delta_{g}
	\end{equation}
	is an orientation cycle for the triple $(C_{c}(G,\mathcal{A}), \widehat{H}, \widehat{D}, \hat{\pi}_{2}\rtimes \hat{\Gamma}, \widehat{J})$ on $A\rtimes_{\alpha,r} G$, where the normalisation factor $M$ is given by
	\begin{displaymath}	M = 
		\begin{dcases*}
		-il(g)(n+1) & if $(A,H,D)$ is odd\\
		l(g)(n+1) & if $(A,H, D)$ is even.
		\end{dcases*}
	\end{displaymath}
Indeed, since
\begin{displaymath}
\begin{split}
c\rtimes_{\alpha} \Delta_{g} &= \sum(a_{0}\delta_{g^{-1}}\otimes  b_{0})\otimes \delta_{g} \otimes a_{1} \otimes \cdots \otimes a_{n} \\
& \qquad   + \sum_{j=2}^{n}(-1)^{j-1}\sum(a_{0}\delta_{g^{-1}}\otimes  b_{0})\otimes \alpha_{g}(a_{1})\otimes \cdots \otimes \alpha_{g}(a_{j-1})\otimes \delta_{g}\otimes a_{j}\otimes \cdots\otimes a_{n} \\
& \qquad \qquad \qquad  + (-1)^{n}\sum(a_{0}\delta_{g^{-1}}\otimes  b_{0})\otimes \alpha_{g}(a_{1})\otimes \cdots \otimes \alpha_{g}(a_{n})\otimes \delta_{g}
\end{split}
\end{displaymath}
we have that 
\begin{displaymath}
\begin{split}
\pi_{\widehat{D}}(c\rtimes_{\alpha} \Delta_{g}) &= \sum\hat{\pi}(a_{0})\hat{\Gamma}^{*}_{g}\widehat{J}\hat{\pi}(b^{*}_{0})\widehat{J}^{-1}[\widehat{D},\hat{\Gamma}_{g}][\widehat{D},\hat{\pi}(a_{1})]\cdots [\widehat{D},\hat{\pi}(a_{n})]\\
& \qquad   + \sum_{j=2}^{n}(-1)^{j-1}\sum\hat{\pi}(a_{0})\hat{\Gamma}^{*}_{g}\widehat{J}\hat{\pi}(b^{*}_{0})\widehat{J}^{-1}[\widehat{D},\hat{\pi}(\alpha_{g}(a_{1}))]\cdots \\
& \hspace{4cm} \cdots [\widehat{D},\hat{\pi}(\alpha_{g}(a_{j-1}))][\widehat{D},\hat{\Gamma}_{g}][\widehat{D},\hat{\pi}(a_{j})]\cdots [\widehat{D},\hat{\pi}(a_{n})]\\
& \qquad  + (-1)^{n}\sum\hat{\pi}(a_{0})\hat{\Gamma}^{*}_{g}\widehat{J}\hat{\pi}(b^{*}_{0})\widehat{J}^{-1}[\widehat{D},\hat{\pi}(\alpha_{g}(a_{1}))]\cdots [\widehat{D},\hat{\pi}(\alpha_{g}(a_{n}))][\widehat{D},\hat{\Gamma}_{g}]
\end{split}
\end{displaymath}
with a slight abuse of notation ($\hat{\pi}$ denotes two copies of the representation $\hat{\pi}_{2}$ and $\hat{\Gamma}$ two copies of $\hat{\Gamma}$).\\

Consider the case when $(\mathcal{A},H,D)$ is odd and thus $(C_{c}(G,\mathcal{A}), \widehat{H}, \widehat{D}, \hat{\pi}_{2}\rtimes \hat{\Gamma}, \widehat{J})$ is even. 
Then
\begin{displaymath}
[\widehat{D}, \hat{\pi}(a)]  = [D,\pi(a)]\otimes 1\otimes \sigma_{1}, \qquad \hat{\Gamma}^{*}_{g}[\widehat{D}, \hat{\Gamma}_{g}]  = l(g)\otimes 1\otimes \sigma_{2},
\end{displaymath}
for any $a\in\mathcal{A}$ and $g\in G$. In particular $\hat{\Gamma}^{*}_{g}[\widehat{D}, \hat{\pi}(\alpha_{g}(a))]\hat{\Gamma}_{g}  = [D,\pi(a)]\otimes 1\otimes \sigma_{1} $ and so 
\begin{displaymath}
\begin{split}
\pi_{\widehat{D}}(c\rtimes_{\alpha} \Delta_{g})(\xi\otimes \delta_{x}\otimes v) &= l(g )\sum\pi(a_{0})J\pi(\alpha_{x}^{-1}(b_{0}^{*}))J^{-1}[D,\pi(a_{1})]\cdots [D,\pi(a_{n})]\xi\otimes \delta_{x}\\
& \qquad \qquad \otimes  \left(\sigma_{2}\sigma_{1}^{n} + \sum_{j=2}^{n}(-1)^{j-1}\sigma_{1}^{j-1}\sigma_{2}\sigma_{1}^{n-j+1} + (-1)^{n}\sigma_{1}^{n}\sigma_{2}\right)v
\end{split}
\end{displaymath}
 by the zeroth order condition of $\widehat{J}$. The summation in the brackets is just $(n+1)$ times $\sigma_{2}\sigma_{1}^{n}$ which is $-i(n+1)\sigma_{3}$ as $n$ is odd, by the properties of the algebra of Pauli matrices. The factor
\begin{equation}\label{362}
	\sum\pi(a_{0})J\pi(\alpha_{x}^{-1}(b_{0}^{*}))J^{-1}[D,\pi(a_{1})]\cdots [D,\pi(a_{n})]
\end{equation}
is just $\pi_{D}(c_{x^{-1}})$ (according to the notation of \eqref{361}). By Lemma \ref{360} this is $\textup{Ad}u_{x}^{*}\circ \pi_{D}(\alpha_{x}^{-1}(c))$. 
Since $c$ is $G$-invariant and $\pi_{D}(c)= \textup{id}_{H}$ we deduce that \eqref{362} is trivial. 
Then since $\sigma_{3}=\rchi$, the normalisation factor brings the thesis.

Consider now the case when $(\mathcal{A},H,D)$ is even and thus $(C_{c}(G,\mathcal{A}), \widehat{H}, \widehat{D}, \hat{\pi}_{2}\rtimes \hat{\Gamma}, \widehat{J})$ is odd. 
Then 
\begin{displaymath}
		[\widehat{D}, \hat{\pi}(a)] = [D,\pi(a)]\otimes 1, \qquad 	\hat{\Gamma}^{*}_{g}[\widehat{D}, \hat{\Gamma}_{g}]  = l(g)(\rchi\otimes 1),
\end{displaymath}
for any $a\in\mathcal{A}$ and $g\in G$. In particular, $	\hat{\Gamma}^{*}_{g}[\widehat{D}, \hat{\pi}(\alpha_{g}(a))]\hat{\Gamma}_{g} = [D,\pi(a)]\otimes 1 $. Since
\begin{displaymath}
	\rchi[D,\pi(a_{j})] = -[D,\pi(a_{j})]\rchi
\end{displaymath} 
we get
\begin{displaymath}
	\begin{split}
		\pi_{\widehat{D}}(c\rtimes_{\alpha} \Delta_{g}) &= l(g )\sum\pi(a_{0})\rchi J\pi(b_{0}^{*})J^{-1}[D,\pi(a_{1})]\cdots [D,\pi(a_{n})]\otimes\\
		& \qquad \qquad \otimes \left(1 + \sum_{j=2}^{n}(-1)^{j-1}(-1)^{j-1} + (-1)^{n}(-1)^{n}\right)\\
		&= l(g)(n+1)\rchi\underbrace{\left(\sum\pi(a_{0}) J\pi(b_{0}^{*})J^{-1}[D,\pi(a_{1})]\cdots [D,\pi(a_{n})]\right)}_{\rchi}\otimes 1
	\end{split}
\end{displaymath}
by the zeroth order condition for $\widehat{J}$. Since 
$\rchi^{2}=\textup{id}_{H}$ by assumption, the normalisation factor $M$ completes the proof.
\end{proof} 

\begin{rmk}
	The  method of the twisted shuffle product is not suitable for the case of the $\star$-equivariance: indeed, the shuffle product sums the degree of the Hochschild chains that are multiplied and it is not possible to pass from dimension $n$ to dimension $n-1$ (apart from multiplying by a hypothetical $7$-cycle whose existence is not guaranteed).  
\end{rmk}

\begin{rmk}Applying Theorem \ref{59} and formula \eqref{chat} to the triple in Example \ref{nctorus}, one recovers (up to a multiplicative constant) the standard orientation cycle on the noncommutative $2$-torus as described in \cite[Chapter 12.3]{gracia2013elements}. Indeed, if we regard $C(S^{1})$ as the $C^{*}$-algebra generated by $U=e^{2\pi i \varphi_{1}}$ with $\varphi_{1}\in S^{1}$, then $c = U^*\otimes U$ is a Hochschild orientation $1$-cycle for the spectral triple over $C(S^{1})$ as defined in Example \ref{nctorus}. Doing the shuffle product with the $1$-cocycle $\delta=V^{*}\otimes V$ (where $V$ is the generator of the action of $\mathbb{Z}$) we have by definition
	\begin{displaymath}
	\begin{split}
			c\rtimes \delta &= U^{*}V^{*}\otimes V\otimes U - U^{*}V^{*}\otimes \alpha(U)\otimes V \\
			&= U^{*}V^{*}\otimes V\otimes U - e^{2\pi i \theta}U^{*}V^{*}\otimes U\otimes V\\
			&= U^{*}V^{*}\otimes V\otimes U - V^{*}U^{*}\otimes U\otimes V
	\end{split}
	\end{displaymath}
	\demo

\end{rmk}

After having constructed an orientation cycle $\hat{c}$ on the equivariant triple $(C_{c}(G,\mathcal{A}),\widehat{H},\widehat{D})$,
we now examine its equivariance for the coaction of $G$ in a suitable sense. 
\begin{defn}
	Let $\delta\colon B \rightarrow B\otimes C^{*}_{r}(G)$ be a coaction of $G$ on a unital $C^{*}$-algebra $B$.
	For any Hochschild chain $c=\sum(b_{0}\otimes p)\otimes b_{1}\otimes \cdots\otimes b_{n}\in C_{n}(B\otimes B^{\textup{op}}, B)$, we define 
	\begin{equation}
		\delta(c) \coloneqq \sum(b_{0 (-1)}\otimes p)\otimes b_{1(-1)}\otimes \cdots\otimes b_{n(-1)}\otimes \left(b_{0(0)}\cdots b_{n(0)}\right)\in C_{n}(B\otimes B^{\textup{op}}, B)\otimes C^{*}_{r}(G)
	\end{equation}
using the Sweedler notation \eqref{51}. We say that $c$ is \emph{$\delta$-invariant} if $\delta(c)=c\otimes 1$. 
\end{defn}

 \begin{prop}
 Let $G$ be a discrete group and $l\colon G\rightarrow \mathbb{R}$ a proper homomorphism. Let $(\mathcal{A}, H, D,u)$ be an (even or odd) $G$-invariant spectral triple on a unital $C^{*}$-algebra $A$, $J$ a unitarily invariant real structure and $c$ a $G$-invariant orientation cycle. The orientation cycle $\hat{c}$ of $(C_{c}(G,\mathcal{A}), \widehat{H},\widehat{D})$ given in \eqref{chat} is invariant for the dual coaction 
 $\widehat{\alpha}$. 
\end{prop}

\begin{proof}
Let us write $c\rtimes_{\alpha} \Delta_{g}=\sum_{j=1}^{n+1}c_{j}$ as a shorthand notation and recall that by definition $\widehat{\alpha}(a\delta_{g})=a\delta_{g}\otimes \delta_{g}$ for any 
$a\delta_{g}\in C_{c}(G,\mathcal{A})$. 
 Concerning $c_{1} = \sum(a_{0}\delta_{g^{-1}}\otimes  b_{0})\otimes \delta_{g} \otimes a_{1} \otimes \cdots \otimes a_{n}$ we have by definition 
$$
		\widehat{\alpha}(c_{1})  = \sum \left((a_{0}\delta_{g^{-1}}\otimes  b_{0})\otimes \delta_{g} \otimes a_{1} \otimes \cdots \otimes a_{n}\right)\otimes (\delta_{g^{-1}}\delta_{g}\delta_{e}\cdots \delta_{e})\\
= c_{1}\otimes \delta_{e}.
$$

The other cases are similar; any factor $a\delta_{e}\in C_{c}(G,\mathcal{A})$ of $c_{j}$ brings a trivial contribution to the piece in $C^{*}_{r}(G)$.
Since any $c_{j}$ contains precisely one term $\delta_{g}$ and one term $\delta_{g^{-1}}$, the total contribution is trivial. 
\end{proof}

\begin{rmk}
If $G$ is abelian, then the orientation cycle $\hat{c}$ given in \eqref{chat} is indeed invariant under the dual action $\widetilde{\alpha}$ of the Pontryagin group $\widehat{G}$ as in Definition \ref{83}. 
\end{rmk}

%---------------------------------------------------------------------------------------------
\appendix
\addtocontents{toc}{\protect\setcounter{tocdepth}{1}}

\section{Equivariant KK-Theory and The Kasparov Descent}\label{appA}

In this appendix we describe the KK-theory of algebras which carry an action of a locally compact topological group $G$ and give a survey of the most important properties. The theory is contained in \cite{kasparov1980operator, gg1988equivariant}; we refer also to \cite[\S 20]{blackadar1998k} and \cite{echterhoff2017bivariant} for more details.

\subsection{Equivariant KK-Theory}\indent 

Let $(B,G, \beta)$ be a $C^{*}$-dynamical system and $E$ a right Hilbert $B$-module. If the action $\beta$ is norm-continuous, we say that $B$ is a \emph{$G$-$C^{*}$-algebra}. An action of $G$ on $E$ is a homomorphism from $G$ into the space of bounded linear transformation on $E$ (not necessarily the space of module homomorphisms $\mathcal{L}_{B}(E)$) such that:
\begin{enumerate}
	\item it is continuous in the strong operator topology, that is the map 
	$g\mapsto \norm{\langle g\cdot x, g\cdot x\rangle_{B} }$ is continuous for every $x\in E$. 
	\item the $G$-action is compatible with the $B$-action in the sense that $g\cdot (xb) = (g\cdot x)\beta_{g}(b)$ for any $g\in G$, $x\in E$ and $b\in B$. 
\end{enumerate}
A Hilbert $B$-module with a (continuous) action $\gamma$ of $G$ is called an \emph{equivariant Hilbert $B$-module}. 
If $E_{1}$ and $E_{2}$ are Hilbert $B$-modules with a $G$-action, there is a natural induced action of $G$ on $\mathcal{L}(E_{1}, E_{2})$ given by conjugation. In general, this action $g\mapsto g\cdot T$ for $T\in \mathcal{L}(E_{1}, E_{2})$ is not norm-continuous: we say that $T$ is \emph{$G$-continuous} if it is. Obviously every $G$-equivariant map is $G$-continuous. In the case of a graded $C^{*}$-algebras and graded Hilbert modules, we require that the action of the group $G$ preserves the subspaces of homogeneous elements.

\begin{defn}
	Let $A$ and $B$ graded $G$-algebras. A \emph{$G$-equivariant Kasparov $A-B$-module} is a triple $(E,\phi,F)$ where $E$ is a countably generated Hilbert $B$-module with an action of $G$, $\phi\colon A\rightarrow \mathcal{L}_{B}(E)$ is an equivariant graded $*$-homomorphism and $F$ is an even $G$-continuous operator in $\mathcal{L}_{B}(E)$ such that
	\begin{equation}\label{cose}
		[F,\phi(a)], \, (F^{2} - 1)\phi(a), \, (F-F^{*})\phi(a)\,\,\, \textup{and}\,\,\, (g\cdot F - F)\phi(a)
	\end{equation}
	are compact operators on $E$ for all $a\in A$ and $g\in G$. 
\end{defn}

$KK^{G}(A,B)$ is the space of $G$-equivariant Kasparov $A-B$-module up to homotopy (as defined in the non equivariant case) and we define $KK_{G}^{n}(A,B)\coloneqq  KK_{G}(A,B\hat{\otimes}\mathbb{C}l_{n})$ where $\mathbb{C}l_{n}$ is the complex Clifford algebra.

Also in the equivariant case there exists an intersection product: given the $G$-algebras $A_{1}, A_{2},B_{1}, B_{2}$  and $D$, there exists a bilinear pairing
\begin{displaymath}
	\hat{\otimes}_{D} \colon KK^{G}_{m}(A_{1}, B_{1}\hat{\otimes}D)\times KK_{n}^{G}(D\hat{\otimes}A_{2}, B_{2})\longrightarrow KK^{G}_{n+m}(A_{1}\hat{\otimes} A_{2}, B_{1}\hat{\otimes}B_{2})
\end{displaymath}
which is associative and functorial in all possible senses. 

\subsection{The Kasparov Descent}\indent

The Kasparov descent map allows to relate the equivariant KK-theory of two algebras with the equivariant KK-theory of their crossed products. Let $(B,G,\beta)$ a $C^{*}$-dynamical system and $E$ a right Hilbert $B$-module. The algebra $C_{c}(G,B)$ acts on $C_{c}(G,E)$ by
	\begin{equation}\label{app1}
	( x\cdot f)(t) = \int_{G}x(s)\beta_{s}\left(f(s^{-1}t)\right)ds
\end{equation}
for $f\in C_{c}(G,B)$ and $x\in C_{c}(G,E)$. We define the crossed product $E\rtimes_{\beta} G$ as the right Hilbert $B\rtimes_{\beta}G$-module obtained by completing the right $C_{c}(G,B)$-module $C_{c}(G,E)$ with respect to the $C_{c}(G,B)$-valued scalar product 
\begin{equation}\label{app2}
	\langle x,y\rangle(t) = \int_{G} \beta_{s^{-1}}\left(\langle x(s), y(st)\rangle_{B}\right) ds, \qquad x,y\in C_{c}(G,E).
\end{equation}

Suppose now to have a covariant Kasparov module $(E,\phi, F)\in \mathbb{E}_{G}(A,B)$. Obviously, the action $\phi\colon A \rightarrow \mathcal{L}_{B}(E)$ induces a left action $\psi\colon A\rtimes_{\alpha} G \longrightarrow \mathcal{L}(E\rtimes_{\beta} G)$ by 
\begin{equation}\label{app3}
	\qquad (\psi(a)x)(t) = \int_{G} \phi(a(s)) [s\cdot x(s^{-1}t)]ds
\end{equation}
for $a\in C_{c}(G,A)$ and $x\in C_{c}(G,E)$. Endowed with this action and the operator $\widetilde{F}\in \mathcal{L}(E\rtimes_{\beta}G)$ defined by
\begin{displaymath}
	(\widetilde{F}x)(t) = F(x(t)) \qquad x\in C_{c}(G,E),
\end{displaymath}
one can show that $(E\rtimes_{\beta}G, \psi, \widetilde{F})$ becomes a $A\rtimes_{\alpha}G-B\rtimes_{\beta}G$ Kasparov module. 

\begin{thm}[Kasparov Descent \cite{kasparov1980operator}]	Let $(B,G,\beta)$ a $C^{*}$-dynamical system  and $(E,\phi, F)\in \mathbb{E}_{G}(A,B)$.  The map $J_{G}$ sending the equivariant Kasparov module $(E,\phi,F)$ to $(E\rtimes_{\beta}G, \psi, \widetilde{F})$ induces a homomorphism of groups
		\begin{displaymath}
			J_{G}\colon KK^{G}(A,B)\longrightarrow KK(A\rtimes_{\alpha}G ,B\rtimes_{\beta}G)
		\end{displaymath}which is functorial in $A$ and $B$ and is compatible with the intersection product. Furthermore, when $A=B$, the map $J_{G}$ is unital in the sense that $J_{G}(\mathbf{1}_{A}) = \mathbf{1}_{A\rtimes_{\alpha}G}$. 

\end{thm}

An analogous statement holds true also for reduced crossed products. Note that this construction works also for unbounded modules as it leaves the operator essentially untouched: an easy computation shows indeed that the Kasparov descent construction commutes with the bounded transform.

\color{black}

\bibliography{bibliography}{}
\bibliographystyle{plain}

\vspace{0.5cm}

\end{document}